\documentclass[10pt]{amsart}
\usepackage[utf8]{inputenc}
\usepackage[T1]{fontenc}

\usepackage[pdftex]{graphicx}
\usepackage{latexsym}
\usepackage{amsmath,amssymb,amsthm}
\usepackage{mathtools}
\usepackage{enumitem}
\usepackage{faktor}
\usepackage{enumerate}
\usepackage{hyperref}
\usepackage{tikz}
\usepackage{xcolor}

\newtheorem{theorem}{Theorem}[section]
\newtheorem{corollary}[theorem]{Corollary}
\newtheorem{proposition}[theorem]{Proposition}
\newtheorem{lemma}[theorem]{Lemma}
\theoremstyle{definition}
\newtheorem{definition}[theorem]{Definition}
\newtheorem*{remark}{Remark}
\newtheorem{example}[theorem]{Example}

\newcommand{\F}{{F}}

\renewcommand{\bullet}{{\circ}}
\newcommand{\Q}{\mathcal{Q}}

\newcommand{\Hmm}[1]{\leavevmode{\marginpar{\tiny%
			$\hbox to 0mm{\hspace*{-0.5mm}$\leftarrow$\hss}%
			\vcenter{\vrule depth 0.1mm height 0.1mm width \the\marginparwidth}%
			\hbox to 0mm{\hss$\rightarrow$\hspace*{-0.5mm}}$\\\relax\raggedright #1}}}
	
	\newcommand{\eat}[1]{}

\begin{document}
	\author[P. Bartmann, M. Keller]{Philipp Bartmann, Matthias Keller}
	
	\address{Philipp Bartmann, Matthias Keller: Institut f\"ur Mathematik, Universit\"at Potsdam
		14476  Potsdam, Germany}
	\email{philipp.bartmann@uni-potsdam.de}
	\email{matthias.keller@uni-potsdam.de}

	\title[]{On Hodge Laplacians on General Simplicial Complexes}

\begin{abstract}

  \noindent We study Laplacians on general countable weighted simplicial complexes from a conceptual point of view. These operators will first be introduced formally before showing that those formal operators coincide with self-adjoint realizations of operators arising from quadratic forms. A major conceptual perspective is the correspondence to signed Schrödinger operators unveiling the Forman curvature. The main results are criteria for essential self-adjointness via lower bounded Forman curvature and a Gaffney type result via completeness. Finally, we study spectral relations between these Laplacians.
\vspace{0.95cc}
\end{abstract}
\maketitle
\section{Introduction} \label{form}

The study of simplicial complexes has a long-standing tradition in mathematics. As a tool in topology, they serve as triangulations of manifolds, discrete objects that encode topological information. Moreover, Laplacians on simplicial complexes are discrete analogues of Laplacians on differential forms on manifolds and can be considered higher dimensional generalizations of graph Laplacians.

Traditionally, they were studied on finite complexes, going back to the work of Eckmann \cite{Eckmann}, who established a theory which stands in analogy to the Hodge theory on compact manifolds (see the appendix by Dodziuk in \cite{Chavel}). Since then, they have received substantial attention. For instance, in \cite{Dodziuk} it was shown that discrete Laplacians on triangulations can be used to approximate the Laplacian on a given Riemannian manifold, thus combining the discrete and continuous setting.

Similar to graphs, there is often no canonical choice of weights, which provide a notions of volume. For example, one may
count the number of vertices in a set which arises from the constant weight one on the vertices, or one can count the number of edges in a set which arises from choosing the vertex degree as weight. Whenever one is faced with a non-canonical choice, it is often advisable to study all possible weights at once. This has proven to be successful for graphs, which have seen a vibrant interest with numerous results; see, e.g., the recent monographs \cite{Barlow,JP,KLW,KN} and references therein. 

For finite weighted simplicial complexes, there is the seminal work \cite{HJ} that studies various spectral properties of the corresponding operators. However, there has also been a substantial push to study infinite simplicial complexes. For example, in relation to topology \cite{F}, group theory \cite{ballmann1997l2,oppenheim2015property,oppenheim2015vanishing}, random walks \cite{LLP,Parzanchevski,Ros}, operator theory \cite{C,BH,GP,Masamune} or heat kernel asymptotics \cite{HL}. Although substantial progress has been made in various special cases, a conceptual study of infinite simplicial complexes seems to be missing. This is the purpose of the present paper.

We study Laplacians on general countable weighted simplicial complexes. This includes, in essence, all the settings above and allows for an infinite vertex set, infinite dimension and non-locally finite complexes. Our starting point is axiomatic with an operator $\delta$ for which we ask
\begin{equation*}
	\delta\delta=0
\end{equation*}
and a basic compatibility assumption with the combinatorics of the simplicial complex, see Definition~\ref{coboundary}. We choose a weight function $m$ on the simplicial complex to introduce $\partial$ as a formal adjoint of $\delta$. Although in most relevant cases we also have $\partial\partial =0$, this is no longer guaranteed in full generality (as was already noticed in \cite{Parzanchevski}), which shows that care has to be taken. An important aspect is that we define $\delta$ and $\partial$ on functions rather than forms. Although this seems unnatural at first glance, we show that our approach is indeed equivalent in Section~\ref{altern}. Yet, our approach has some advantages; especially it allows for a more convenient notation. Similar ideas have already been pursued elsewhere, \cite{BSSS,HK,JM}.

For the operators $\delta$ and $\partial$, we establish a Stokes' theorem and Green's formulas on appropriate function spaces, which form the backbone of our analysis in Section~\ref{stokes}. In order to study self-adjoint realizations of these operators, we take the approach via quadratic forms. To this end, we make a mild assumption on the weight function $m$ called \textit{local summability} which means that for all simplices $\tau\neq\varnothing $ in the complex 
\begin{equation*}
	\sum_{\sigma\succ \tau}m(\sigma)<\infty,
\end{equation*}
where the sum runs over all $\sigma$ which have $\tau$ as a face, see also e.g. \cite{straussinfinite}. This guarantees that finitely supported functions are included in the form domain. We also introduce a stronger condition called \textit{strong local summability} in Section~\ref{summable}, which guarantees $\partial\partial=0$.

When restricting the quadratic forms to $\ell^{2}$, we observe that there is, in general, a whole family of closed forms, Section~\ref{Dir&Neum}. This comes as no surprise, as this appears also in the continuum where the Sobolev spaces $W^{1,2}_{0}$ and $W^{1,2}$ in general do not coincide. This is also a well studied phenomenon in the graph case, see e.g. \cite{JP,KL,KLW,HKMW,Wojc}. These closed forms then give rise to a family of Laplacians $\Delta^{+}$, $\Delta^{-}$ and $\Delta^{H}$ that we show to be restrictions of $\partial\delta$, $\delta\partial$ and $(\delta+\partial)^{2}$ respectively, where the latter is equal to $\partial\delta+\delta\partial$ in most (yet not all) relevant cases.
We then view our approach through the intriguing lens of Hilbert complexes in Section~\ref{Hilbert}, which gives rise to two particular instances of $\Delta^{H},$ see e.g.\cite{BL}. These are of particular interest, as their kernels are tightly connected to the reduced $\ell^2$-cohomology groups of the simplicial complex. 

A major conceptual perspective is then developed in Section~\ref{DiscSch} where we interpret our Laplacians as specific \emph{signed Schrödinger operators}, Theorem~\ref{schröd2}. Here, we take full advantage of the fact that our operators are defined on functions rather than forms.
We highlight two conceptual benefits which arise from this perspective: First of all, this representation unveils a quantity which is known as Forman curvature \cite{F,JM}. This is one of a multitude of notions of discrete Ricci curvature that has been  studied abundantly for graphs in recent years \cite{BHLLMY,EM,HLLY,Mu,MW,O}.
Secondly, from this point on, we can draw from the rich and in many aspects well-developed theory of signed Schrödinger operators on graphs. In that section, Section~\ref{DiscSch}, we already characterize and give criteria when compactly supported functions are in the operator domain before we deal with the uniqueness of the operators itself.

We then come to the part that can be considered the main result of the paper, which is concerned with uniqueness of the operators in Section~\ref{FormeqEssSa}. As a first step, we characterize boundedness of the operators for finite dimensional complexes in terms of boundedness of an explicit function resembling the vertex degree for graphs, Section~\ref{boundedness}.
Then we address the unbounded case, where we study form uniqueness, i.e., existence of only one closed form, and essential self-adjointness, i.e., existence of only one self-adjoint extension of the Laplacians restricted to compactly supported functions. 
We give two explicit criteria for this to hold:
\begin{itemize}
	\item [(1)] Lower bounded Forman curvature and uniformly positive weight, Theorem~\ref{esssa1}.
	\item [(2)] Compactness of distance balls for certain distances on the 1-skeleton, Theorem~\ref{metric}. 
\end{itemize}
The second result can be seen as an analogue of a famous theorem often attributed to Gaffney \cite{G} in the case of manifolds, see also  \cite{Chernoff,Roe,Strichartz}. Furthermore, these considerations extend the known results of \cite{AACT,baloudi2019adjacency,BH,C,Masamune} for simplicial complexes and results in \cite{CTT,GKS,Go,HKMW,KL,Milatovic,S} for graphs.

We then turn to the study of spectral relations between the Laplacians in Section~\ref{specrel}. It is well known for finite simplicial complexes or compact manifolds that the spectra of the operators $\Delta^{+}$, $\Delta^{-}$ and $\Delta^{H}$ coincide apart from zero, \cite{HJ}. However, in case of non-uniqueness of the operators, this is not true anymore as it is well known that an operator with Dirichlet- or Neumann boundary condition can have different spectra already in the graph case, see e.g. \cite{duefel2024boundaryconditionsmatterspectrum}. 
Thus, we establish the relations which still hold in the general setting, Theorem~\ref{spectrum0} and then later recover the expected spectral equality in the case of a unique operator, Theorem~\ref{spectrum1}.

We end the paper by looking at two natural choices of weights. The first is the combinatorial weight, which assigns constant weight equal to one to each simplex. Here, many of the quantities of the paper simplify significantly and become much clearer and combinatorial in nature. Moreover, we want to emphasize that the metric in the Gaffney-type result of Theorem~\ref{metric} can be chosen as an intrinsic metric on the underlying 1-skeleton, which was already used for graphs in \cite{HKMW}, see Theorem~\ref{comb}. Finally, we review choices related to normalizing weights, which were already considered in \cite{HJ} and which are useful to recover certain spectral estimates in the infinite case.

Throughout the paper, we fix the following notation. For a function $f:\Sigma\to \mathbb{R}$ on a countable set $\Sigma$, we denote
\begin{equation*}
\sum_{\Sigma}f=	\sum_{\sigma\in \Sigma}f(\sigma)
\end{equation*}
whenever the sum converges absolutely.  We denote by $1_{x}$ the characteristic function of a singleton set $\{x\}$.

For two linear operators $A,B$ with domains $D(A)$, $D(B)$ recall the definition of the sum $A+B$ and the composition $AB$  which are defined pointwise on their respective domains
\begin{equation*}
	D(A+B)=D(A)\cap D(B)\quad \mbox{and }\quad D(AB)=\{f\in D(B)\mid Bf\in D(A)\}.
\end{equation*}
Furthermore, we write $A\subseteq B$ if $D(A) \subseteq D(B)$ and $A=B$ on $D(A)$.  We use the same notation for two quadratic forms $Q\subseteq Q'$ on a Hilbert space if $D(Q)\subseteq D(Q')$ and $Q=Q'$ on $D(Q)$. The adjoint of an operator $A$ is denoted by $A^*.$\\

\textbf{Acknowledgments.} The authors acknowledge the financial support of the DFG with the priority program ``Geometry at Infinity''. Furthermore, PB is grateful for discussions with Christian Rose and Lior Tenenbaum on various topics within the paper. MK also expresses his gratitude to Bobo Hua, Jun Masamune and Florentin Münch for valuable discussions.

\section{Setting} \label{setting}
\subsection{Basic Definitions}
We define simplicial complexes as well as coboundary $\delta$ and boundary operator $\partial$. The characteristic features $\delta\delta=0$ and $\partial\partial =0$ are built into the definition of $\delta$. However, for $\partial$ this might not hold for all functions on non-locally finite complexes, which is discussed at the end of the section. 

\begin{definition}[Simplicial Complex]\label{SimplCompl}
Let $V$ be a countable set and $\hat\Sigma\subseteq\mathcal{P}(V)$ be a collection of finite subsets of $V.$ We call $\hat\Sigma$ a \textit{simplicial complex} if $\sigma\in\hat\Sigma$ and $\tau\subseteq \sigma$ implies $\tau\in\hat\Sigma.$ Further, we denote by $\hat\Sigma_{k}$ the collection of elements of cardinality $k+1$  and call their elements \textit{simplices} or $k$-\emph{simplices}.
\end{definition}

 Whenever $\tau,\sigma\in\hat\Sigma$ with $\tau\subset\sigma$ and $\vert\sigma\backslash\tau\vert = 1$ we write $\tau\prec\sigma$ or $\sigma\succ \tau$ and say that $\tau$ is a \textit{face} of $\sigma$ or that $\sigma$ is a \textit{coface} of $\tau.$ Simplices that have no cofaces are called \textit{maximal}. We will define the \textit{dimension} of a simplex as $\mathrm{dim}(\tau)=\vert\tau\vert-1$ and the \emph{dimension} of the simplicial complex as
 \begin{equation*}
 	\mathrm{dim}(\hat\Sigma) =\sup_{\tau\in \hat\Sigma}\dim(\tau).
 \end{equation*}
 We call a simplicial complex \emph{locally finite} if for all $\tau\in\hat\Sigma$, $\tau\neq\varnothing$
 \begin{equation*}
 	\#\{\sigma\in\hat\Sigma\mid \sigma\succ \tau\}<\infty.
 \end{equation*}
 Whenever convenient, we will identify $\hat\Sigma_0$ with $V.$ We refer to $\bigcup_{l\leq k}\hat{\Sigma}_l$ as the $k$-\textit{skeleton} of $\hat\Sigma.$\\
 Note that $\varnothing$ is the only element of $\hat\Sigma_{-1}$ and that $\varnothing\prec v$ for every $v\in\hat\Sigma_0.$ This (rather pathological) fact creates difficulties when it comes to defining the proper function spaces on infinite simplicial complexes that are consistent with the established theory on finite complexes as in \cite{GundertWagner,HJ}. 
 We deal with this by introducing a weight function $$m:\hat\Sigma\rightarrow(0,\infty)$$ and defining $$\Sigma = \hat\Sigma\quad\mbox{ if }\sum_{v\in \Sigma_{0}}m(v)<\infty\qquad\mbox{ and }\qquad\Sigma = \hat\Sigma\backslash\lbrace\varnothing\rbrace\;\mbox{ otherwise}.$$ For the rest of the paper we will identify $\Sigma$ with $\hat\Sigma$, even though, strictly speaking, it might not be a simplicial complex in the sense of Definition \ref{SimplCompl}. This is justified since we do not lose any information about $\hat\Sigma$ by considering $\Sigma$ instead. Furthermore, we may extend $m$ to subsets $A\subseteq \Sigma$ via
 \begin{equation*}
 	m(A)=\sum_{\sigma\in A}m(\sigma)
 \end{equation*}
 and in this sense we can think of $m$ as a measure of full support.
 
We define $$F = F(\Sigma) = \mathbb{R}^\Sigma$$
and  define $\F(\Sigma_k)$ analogously.
Naturally, $\F(\Sigma)$ can be identified with $\bigoplus_{k}\F(\Sigma_k).$ 

Next, we introduce the coboundary operator.

\begin{definition}[Coboundary Operator]\label{coboundary}
    A linear map $\delta:D(\delta)=\F\rightarrow \F$ is called a \textit{coboundary operator} if for all $\tau,\sigma\in\Sigma$, we have $\delta 1_\tau(\sigma)\in\lbrace -1,0,1\rbrace,$ $\delta 1_\tau(\sigma)\neq 0$ iff $\tau\prec\sigma$ and 
     $$\delta\delta = 0.$$
\end{definition}
Given $\tau,\sigma\in \Sigma$ we will write
\begin{equation*}
	\theta(\tau,\sigma)=\delta1_\tau(\sigma),
\end{equation*}
where $1_{\tau}$ is the characteristic function of $\tau$. For $\omega\in \F$ and $\sigma\in \Sigma$, we obtain
 $$\delta\omega(\sigma) = \sum_{\tau\prec\sigma}\omega(\tau)\theta(\tau,\sigma)$$
since $\sigma$ only has finitely many faces $\tau\prec\sigma$. Note that $\delta$ does not depend on the choice of the weight $m.$

\begin{remark}
    We may identify $\delta$ with $\bigoplus_k\delta_k,$ where $\delta_k:\F(\Sigma_k)\rightarrow \F(\Sigma_{k+1})$ is the restriction of $\delta$ to $\F(\Sigma_k).$ Most results in this paper concerning $\delta$ and operators derived from it have their respective analogues for $\delta_k$ with conditions that only need to be satisfied on $ \F(\Sigma_k)$ rather than all of $\F$, see e.g. the remark after Theorem~\ref{spectrum0} for an exception. However, we will mostly state them only for $\delta$ for the sake of brevity and in order to avoid repetition. 
\end{remark}

\begin{remark}
    We define the coboundary operator in a nonstandard, more axiomatic way compared with what one usually finds in the literature (see \cite{GundertWagner,HJ,HL,Parzanchevski}). One difference is that we consider functions rather than alternating forms. This has several practical reasons. It allows us to study $\delta$ without the need to introduce oriented simplices. Furthermore, one can easily establish a connection to the theory of discrete signed Schrödinger operators (see Section \ref{DiscSch}). We justify our approach in Subsection \ref{altern}, where we show that $\delta$ can be identified with a coboundary operator on alternating forms and vice versa.  
\end{remark}

Next, we define the spaces
\begin{align*}
	\F^1=	\F^1(\Sigma,m)& = \lbrace\omega\in \F\mid\sum_{\sigma\succ \tau}m(\sigma)\vert\omega(\sigma)\vert<\infty~\text{for all}~\tau\in\Sigma\rbrace,\\
	\F^2=	\F^2(\Sigma,m)& = \lbrace\omega\in \F^1\mid\sum_{\sigma\succ\tau\succ \rho}m(\sigma)\vert\omega(\sigma)\vert<\infty~\text{for all}~\rho\in\Sigma\rbrace,\\
	\F^c=	\F^c(\Sigma)&=\lbrace\omega\in \F\mid \omega \mbox{ has finite support}\rbrace.
\end{align*}
On $D(\partial)=\F^1,$ we define the \textit{boundary operator} $\partial:\F^1\rightarrow \F$ as
$$\partial\omega(\rho) = \frac{1}{m(\rho)}\sum_{\tau\succ \rho}m(\tau)\omega(\tau)\theta(\rho,\tau)$$ for all $\rho\in\Sigma.$ This also means that if $\omega\in F(\Sigma_0)$ and $\varnothing\notin\Sigma,$ then $\partial\omega = 0.$

\begin{remark}
    Similar to $\delta$, we identify $\partial$ with $\bigoplus_k\partial_k,$ where $\partial_k:\F^1(\Sigma_{k+1})\rightarrow \F(\Sigma_k)$ is the restriction of $\partial$ to $$F^1(\Sigma_{k+1}) = \F^1(\Sigma_{k+1},m) = \lbrace\omega\in \F(\Sigma_{k+1})\mid\sum_{\sigma\succ \tau}m(\sigma)\vert\omega(\sigma)\vert<\infty~\text{for all}~\tau\in\Sigma_k\rbrace$$
\end{remark}
 The next lemma collects useful relations between the function spaces.

\begin{lemma}\label{sammellemma1} We have $\F^c\subseteq \F^2\subseteq\F^1\subseteq \F$ as well as
      $$\partial \F^c\subseteq \F^c\quad\mbox{ and }\quad\partial \F^2\subseteq \F^1.$$
    \end{lemma}
\begin{proof}
    Let $\rho,\sigma\in \Sigma$ be fixed and $\omega = 1_{\sigma}.$ Then, 
    $$\sum_{\sigma'\succ\tau\succ \rho}m(\sigma')\vert\omega(\sigma')\vert \leq 2m(\sigma)<\infty,$$ as there are at most two $\tau$
 such that $\sigma\succ \tau\succ\rho.$ This shows $\F^{c}\subseteq \F^{2}$ because the functions $1_{\sigma}$, $\sigma\in \Sigma$, form a basis of $\F^c$. The inclusions $F^2\subseteq F^1\subseteq F$ follow immediately from the definition. The inclusion $\partial F^c\subseteq F^c$ is a simple consequence of the fact that every simplex only has finitely many faces and $\partial F^2\subseteq F^1$ is a direct application of the triangle inequality.
 \end{proof}

\begin{lemma}\label{pp0} We have
$$\partial\partial = 0 \qquad \mbox{on $\F^{2}$}.$$   
\end{lemma}
\begin{proof}
   We start with an observation about $\theta$ which was introduced alongside the definition of $\delta$. Let $\rho\in\Sigma_{k-1},\sigma\in\Sigma_{k+1}$ and $\rho\subset\sigma.$ Then there are exactly two $\tau,\tau'\in\Sigma_k$ such that $\rho\prec\tau,\tau'\prec\sigma.$ Hence,
    $$0= \delta\delta1_\rho(\sigma) = \sum_{\tau\prec\sigma}\theta(\tau,\sigma)\sum_{\rho'\prec\tau}\theta(\rho',\tau)1_\rho(\rho') = \theta(\rho,\tau)\theta(\tau,\sigma) + \theta(\rho,\tau')\theta(\tau',\sigma).$$
    By  Lemma~\ref{sammellemma1}, we have $\partial \F^{2}\subseteq \F^{1}=D(\partial)$. Hence, $\partial\partial$ is defined on $\F^{2}$. With this at hand, we get for $\rho\in\Sigma_{k-1}$
    \begin{multline*}
        m(\rho)\partial\partial\omega(\rho)  = \sum_{\tau\succ \rho}\theta(\rho,\tau)\sum_{\sigma\succ \tau}\theta(\tau,\sigma)m(\sigma)\omega(\sigma)\\
        = \sum_{\sigma\in\Sigma_{k+1}}m(\sigma)\omega(\sigma)\left[\theta(\rho,\tau)\theta(\tau,\sigma) + \theta(\rho,\tau')\theta(\tau',\sigma)\right]1_{\rho\prec\tau,\tau'\prec\sigma}1_{\tau\neq\tau'} = 0.
    \end{multline*}   
    For the second equality, we changed the order of summation with respect to $\tau$ and $\sigma$, which is justified by absolute convergence, due to $\omega\in F^2.$
\end{proof}
The following example shows that in general $\partial\partial = 0$ cannot be expected, not even when the composition is well defined. In principle, this was already observed in \cite{Parzanchevski}, however, only for $\partial_{-1}\partial_{0}.$ Since we allow for non-locally finite simplicial complexes, this phenomenon may now occur in higher dimensions as well.

\begin{example}[$\partial\partial\neq 0$]\label{example}
    Let $V = \mathbb{N}\cup\lbrace 0\rbrace$ and $\Sigma$ be the finite dimensional simplicial complex whose maximal faces are given by the triangles $\tau_n = \lbrace 0,n,n+1\rbrace$, $n\ge 1$, (which fully determines $\Sigma$). Let $\rho_n = \lbrace 0,n\rbrace,n\geq 1.$
    \begin{figure}[h]
    	\centering

    	\label{fig:enter-label}
    	
    	\begin{tikzpicture}
    		
    		\fill[fill=gray!30] (0,0)--(0,2)--(6,2);
    		\fill (0,0) circle (2pt);
    		\draw (0,0)--(0,2)--(5,2);
    		\filldraw[black] (1,2) circle (2pt) node[anchor=south]{2};
    		\filldraw[black] (2,2) circle (2pt) node[anchor=south]{3};
    		\filldraw[black] (3,2) circle (2pt) node[anchor=south]{4};
    		\filldraw[black] (4,2) circle (2pt) node[anchor=south]{5};
            \filldraw[black] (5,2) circle (2pt) node[anchor=south]{6};
            \filldraw[black] (6,2) circle (2pt) node[anchor=south]{$n$};
    		\filldraw[black] (0,2) circle (2pt) node[anchor=south]{1};
    		\draw (0,0)--(1,2);
    		\draw (0,0)--(2,2);
    		\draw (0,0)--(3,2);
    		\draw (0,0)--(4,2);
            \draw (0,0)--(5,2);
            \draw (0,0)--(6,2);
            
            \node at (6.5,2) {\scalebox{1.5}{$\cdots$}};
            \node at (5.5,2) {\scalebox{1.5}{$\cdots$}};
            \node at (7, 1){\scalebox{1.5}{$\cdots$}};
    		\coordinate[label=left:0] (A0) at (0,0);

    	\end{tikzpicture}
    	\caption{The simplicial complex from Example \ref{example}}
    \end{figure}
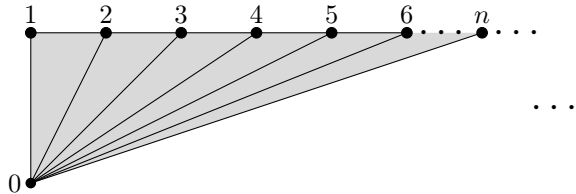
    
     Choose the weight $m$ such that $m(\lbrace n,n+1\rbrace) = n^2$, $ m(\rho_n) = 1\slash n^2$, $m(\tau_n) = n^2,$ for $n\ge 1$ and   $m(\lbrace n\rbrace) = 1$ 
     for $n\in V$. In fact, only $m(\tau_n)$ will matter for the moment. Moreover, for $m,n\in V,m<n$ let $ 1 = \theta (m,\lbrace m,n\rbrace) = -\theta (n,\lbrace m,n\rbrace)$ as well as $ 1 = \theta(\rho_n,\tau_n) = \theta(\lbrace n,n+1\rbrace,\tau_n) = -\theta(\rho_{n+1},\tau_n)$ for $n>0.$ It is then easy to check that $\delta$ acting as $$\delta\omega(\sigma) = \sum_{\tau\prec\sigma}\theta(\tau,\sigma)\omega(\tau)$$ defines a coboundary operator on $\Sigma$ in the sense of Definition \ref{coboundary}. Keep in mind that $\varnothing\notin\Sigma$ here.

     Now consider the function
     $$\omega_0 = 1_{0} + \sum_{n\in\mathbb{N}}\frac{1}{n^2}1_{\tau_n}.$$
     
       Then $\omega_0\in \F^1.$ Actually, $\omega_0$ is even square summable with respect to $m$. Moreover, we show $\partial\omega_0\in \F^1,$ which means that $\partial\partial$ can be applied to $\omega_0$. To see this, we need to check that for $v\in\Sigma_0$
    $$\sum_{\rho\succ v} m(\rho)\vert\partial\omega_0(\rho)\vert<\infty.$$
    However, only for $v_{0}= \lbrace 0\rbrace$ this sum has infinitely many terms so that we only need to check this case. For $n\geq 2$, we see that $\partial\omega_0(\rho_n) = 0$ and so
        $$\sum_{\rho\succ v_{0}}m(\rho)\vert\partial\omega_0(\rho)\vert = \sum_{n\in\mathbb{N}}m(\rho_n)\vert\partial\omega_0(\rho_n)\vert
     = m(\tau_1)\vert\omega_0(\tau_1)\vert = 1<\infty.$$
    One may even show that $\partial\omega_0$ is square summable with respect to $m,$ which is precisely what we will be interested in later on. Despite all this a similar calculation shows $$\partial\partial\omega_0(v_0) = 1.$$
\end{example}

\subsection{Relation to Coboundary Operators on Alternating Forms}\label{altern}

Traditionally, the coboundary operator associated to a simplicial complex is defined as a linear operator acting on alternating forms on oriented simplices. More precisely, given $\tau = \lbrace v_1,\ldots,v_{k+1}\rbrace\in\Sigma,$ one defines oriented simplices arising from $\tau$ as equivalence classes $[(v_{\pi(1)},\ldots,v_{\pi(k+1)})],$ where $\pi$ is a permutation. Two such classes are the same iff the respective permutations have the same parity. In particular, if $k\geq1$ then one always has exactly two oriented simplices associated to $\tau.$ 
With slight abuse of notation, we will denote the two equivalence classes by $[\tau]$ and $\overline{[\tau]}.$ For $v\in \Sigma_{0}$, there is only one equivalence class $[v]$. We denote by $\Sigma^{\mathrm{or}}$ the set of oriented simplices.

An alternating form $\omega$ then is a function $\Sigma^{\mathrm{or}}\to \mathbb{R}$ such that $\omega([\tau]) = -\omega(\overline{[\tau]})$ for all oriented simplices $[\tau].$ The coboundary operator $\Tilde{\delta}$ is then defined as 
\begin{equation}\label{cobform}
    \Tilde{\delta}\omega([v_1,\ldots,v_{k+1}]) = \sum_{i = 1}^{k+1}(-1)^{j+1}\omega([v_1,\ldots,\hat{v}_j,\ldots,v_{k+1}])\tag{$\ast$},
\end{equation}
where $\hat{v}_j$ means that the vertex $v_j$ is omitted. It is easy to show that $\Tilde{\delta}\Tilde{\delta} = 0.$ See \cite{HJ,HL,Parzanchevski} for more details.\\
In contrast to Definition \ref{coboundary} we consider oriented simplices and alternating forms instead of simplices in $\Sigma$ and functions acting on them. We show that both settings are essentially equivalent and differences are merely of a cosmetic nature. 

Let $\Omega(\Sigma^{\mathrm{or}})$ be the space of alternating forms. For each $\tau\in\Sigma$, pick one of the (at most) two associated oriented simplices $[\tau]$ and $\overline{[\tau]}$ and call it $\tau^{\mathrm{or}}.$ This way we can identify alternating forms with functions on $\Sigma.$ More precisely, the map 
$T:\Omega(\Sigma^{\mathrm{or}})\rightarrow \F(\Sigma)$ with
$$T\omega(\tau) = \omega(\tau^{\mathrm{or}})$$ is a bijection with inverse $$T^{-1}f(\tau^{\mathrm{or}}) = f(\tau)=-T^{-1}f(\overline{\tau^{\mathrm{or}}}).$$

\begin{theorem}The operator $T\Tilde{\delta} T^{-1}$ is a coboundary operator satisfying Definition~\ref{coboundary}. In particular, on every simplicial complex there is always such a coboundary operator. 
\end{theorem}
\begin{proof}
    The proof is straight forward.
\end{proof}

It may be less obvious that every operator satisfying Definition~\ref{coboundary} gives rise to a coboundary operator as in (\ref{cobform}).\\
Let $[\tau] = [v_1,\ldots,v_{k+1}]$ be any oriented simplex and $\tau = \lbrace v_1,\ldots,v_{k+1}\rbrace$ the associated simplex in $\Sigma.$ Let $\tau_i = \lbrace v_1,\ldots,v_i\rbrace$ be the truncation of $\tau$ to the vertices $v_1,\ldots,v_i$, $1\le i\le k+1$. For the coboundary operator $\delta$ as in Definition \ref{coboundary} and $\theta$ correspondingly, we  define
$$h:\Sigma^{\mathrm{or}}\to \{\pm 1\},\qquad h([\tau]) = \prod_{i = 1}^{k}\theta(\tau_i,\tau_{i+1}).$$
The next lemma shows that $h$ is an alternating form and it is therefore independent of the choice of the representative of $[\tau]=[v_{1},\ldots,v_{k+1}],$ i.e. it is well-defined.

\begin{lemma}\label{h}
   The map $h$ is an alternating form. If $[\sigma] = [v_1,\ldots,v_{k+1}]$, $[\tau] = [v_1,\ldots,\hat{v}_i,\ldots,v_{k+1}],\sigma = \lbrace v_1,\ldots,v_{k+1}\rbrace$ and $\tau = \sigma\backslash\lbrace v_i\rbrace,$ then $$h([\sigma]) = (-1)^{k-i+1}h([\tau])\theta(\tau,\sigma).$$
\end{lemma}
\begin{proof}
    In order to show that $h$ is alternating, 
     we recall that whenever $\rho\prec\tau',\tau''\prec\sigma$ and $\tau'\neq\tau'',$ then 
    $$0 = \delta\delta 1_\rho(\sigma) = \theta(\rho,\tau')\theta(\tau',\sigma)+\theta(\rho,\tau'')\theta(\tau'',\sigma).$$ Choosing $\rho = \tau_{i-1},\tau' = \tau_i,\tau'' = \tau_{i+1}\backslash\lbrace v_{i}\rbrace,\sigma = \tau_{i+1}$, we observe that
    \begin{multline*}
        h([v_1,\ldots,v_i,v_{i+1},\ldots,v_{k+1}])  = \theta(\rho,\tau')\theta(\tau',\sigma)\cdot\prod_{\substack{j = 1\\j\neq i-1,i}}^{k}\theta(\tau_j,\tau_{j+1})\\
        = -\theta(\rho,\tau'')\theta(\tau'',\sigma)\cdot\prod_{\substack{j = 1\\j\neq i-1,i}}^{k}\theta(\tau_j,\tau_{j+1})
        = -h([v_1,\ldots,v_{i+1},v_i,\ldots,v_{k+1}]).
    \end{multline*}
    This shows that $h$ is alternating. In particular, this shows that $h$ is well-defined, i.e., independent of the choice of the representative of $[\tau]$.     
    With this at hand, the formula follows immediately
    \begin{equation*}
    	h([\sigma]) = (-1)^{k+1-i}h([v_1,\ldots,\hat{v}_i,\ldots,v_{k+1},v_i]) = (-1)^{k-i+1}h([\tau])\theta(\tau,\sigma).\hfill\qedhere
    \end{equation*}
\end{proof}
Given the form $h$, we define a map $U:\Omega(\Sigma^{\mathrm{or}})\rightarrow \F(\Sigma)$ as $$U\omega(\tau) = h([\tau])\omega([\tau]),$$ which is well defined and bijective.  The inverse is given by $$U^{-1}f([\tau]) = h([\tau])f(\tau).$$
\begin{theorem}
    Let $\delta$ be a coboundary operator as in Definition \ref{coboundary}. Then 
    $$\Tilde{\delta}\omega([v_1,\ldots,v_{k+1}]) = (-1)^k U^{-1}\delta U\omega([v_1,\ldots,v_{k+1}])$$  satisfies (\ref{cobform}).
\end{theorem}
\begin{proof}
    The proof is a direct application of Lemma \ref{h}.
\end{proof}

Our final observation is that $U$ is unitary when restricted to the  Hilbert space
$$\Omega^2(\Sigma^{\mathrm{or}}) = \Omega^2(\Sigma^{\mathrm{or}},m^{\mathrm{or}}) = \lbrace\omega\in\Omega(\Sigma^{\mathrm{or}})\mid \sum_{[\tau]\in\Sigma^{\mathrm{or}}}m^{\mathrm{or}}([\tau])\vert\omega([\tau])\vert^2<\infty\rbrace,$$
where the weight is given by $m^{\mathrm{or}}([\tau]) = m(\tau)\slash 2$ for $\tau\in\Sigma\setminus \Sigma_0$ and $m^{\mathrm{or}}([v]) = m(v)$ if $v\in\Sigma_0.$ Indeed, $U$ maps $\Omega^{2}(\Sigma^{\mathrm{or}})$
 onto  $\ell^2(\Sigma,m)= \lbrace\varphi\in \F\mid \sum_{\Sigma}m\varphi^2<\infty\rbrace.$ The fact that the restriction of $U$ to $\Omega^2(\Sigma^{\mathrm{or}})$ is unitary is now obvious, as it is in essence multiplication with a sign.
In consequence it can be shown that suitable restrictions of $\Tilde{\delta}$ and $\delta$ are unitarily equivalent by the theorem above.

\subsection{Stokes's Theorem and Green's Formulas}\label{stokes}

It is well known that for finite simplicial complexes the boundary and coboundary operators are adjoints of each other. In this section, we provide a first treatment of this phenomenon in the infinite setting, which can be understood as a discrete version of Stokes' theorem. 

\begin{theorem}[Stokes' theorem]\label{adj1}
    For $\omega\in \F^c$ and $\eta\in \F$ or $\omega\in \F^1$ and $\eta\in \F^c,$ 
    $$\sum_{\Sigma}m(\partial\omega)\eta =\sum_{\Sigma}m\omega(\delta\eta).$$
\end{theorem}
\begin{proof}
    The statement follows from Fubini's theorem as in either case the sum
    $\sum_{\sigma}\sum_{\tau\prec\sigma} m(\tau)\vert\omega(\tau)\vert\vert\eta(\sigma)\vert$ 
    converges.    \end{proof}

\begin{corollary}\label{adj2}
    Let $\omega\in \F^c$ and $\eta\in \F$ or $\omega\in \F^2$ and $\eta\in \F^c.$ Then 
    $$ \sum_{\Sigma}m(\partial \omega)(\delta \eta) =0.$$
\end{corollary}
\begin{proof}
     By Lemma~\ref{sammellemma1}, $\partial \F^c\subseteq \F^c$ as well as $\partial \F^2\subseteq\F^1.$
    The statement now follows from Theorem \ref{adj1} and Lemma \ref{pp0} as either $\delta\delta\eta = 0$ or $\partial\partial\omega = 0.$
\end{proof}

The  Green's formulas below will be of major importance in the following.

\begin{theorem}[Green's formulas]\label{green1}
	\begin{itemize}
		\item [(a)] If $\omega\in \F^1$ and $\eta\in \F^c,$ then
		$$\sum_{\Sigma}m(\delta\partial\omega)\eta= \sum_{\Sigma}m(\partial\omega)(\partial\eta)= \sum_{\Sigma}m\omega(\delta\partial\eta).$$   
		\item [(b)] If $\omega\in D(\partial\delta) = \lbrace\varphi\in \F\mid\delta\omega\in \F^1\rbrace$ and $\eta\in \F^c,$ then 
		$$\sum_{\Sigma}m(\partial\delta\omega)\eta= \sum_{\Sigma}m(\delta\omega)(\delta\eta).$$
		\item [(c)] If $\omega\in F^{+}= \lbrace\varphi\in \F\mid\sum\limits_{\sigma\succ\tau}m(\sigma)\sum\limits_{\tau'\prec\sigma}\vert\varphi(\tau')\vert<\infty,~\tau\in\Sigma\rbrace$ and $\eta\in \F^c,$ then
		$$\sum_{\Sigma}m(\partial\delta\omega)\eta= \sum_{\Sigma}m\omega(\partial\delta\eta).$$
	\end{itemize}  
\end{theorem}
\begin{proof}
Statement $ (a)$ and $(b)$ follow from Theorem \ref{adj1} using $\partial\eta\in \F^c$ whenever $\eta\in \F^c$ by Lemma~\ref{sammellemma1}. Statement $(c)$ follows from Fubini's theorem, where $\omega\in D(\partial\delta)$ is implied by the triangle inequality.
\end{proof}

\section{Quadratic Forms and Laplacians}\label{quadform}

In this section, we introduce Laplacians on simplicial complexes. We pursue this goal via closed and positive quadratic forms arising from $\delta$, $\partial$ and $\delta+\partial$ and obtain Laplacians as their corresponding self-adjoint operators. We then study their action and find out that they are restrictions of $\partial\delta$, $\delta\partial$ and $(\delta+\partial)^{2}.$ The latter coincides with $\partial\delta+\delta\partial$ on a certain subspace. In general there is no unique self-adjoint realization of these operators and we discuss in particular the so-called Dirichlet and Neumann Laplacian as the two most significant cases. Furthermore, we introduce a Laplacian which naturally arises in the context of Hilbert complexes.

\subsection{General quadratic forms and local summability}\label{summable}
For $\omega\in \F$, we consider
$\|\omega\|=\left(\sum_{\Sigma}m\omega^{2}\right)^{1/2}.$
This defines a norm on   $$\ell^{2}=\ell^2(\Sigma,m) = \lbrace\omega\in \F\mid \|\omega\|<\infty\rbrace,$$
which is a Hilbert space with scalar product $\langle\omega,\eta\rangle = \sum_{\Sigma}m\omega\eta$, $\omega,\eta\in \ell^{2}$.

The Laplacians will be introduced by virtue of their quadratic forms, which we first define on their maximal domains of definition beyond $\ell^{2}$.
Let
\begin{align*}
  \Q^+:\F\rightarrow[0,\infty],&&  \Q^+(\omega) &=\|\delta\omega\|^{2}= \sum_{\Sigma}m\vert\delta\omega\vert^2,\\
  \Q^-:\F^{1}\rightarrow[0,\infty],&&  \Q^-(\omega) &=\|\partial\omega\|^{2}= \sum_{\Sigma}m\vert\partial\omega\vert^2,\\
 \Q^H:\F^{1}\rightarrow[0,\infty],&&   \Q^H(\omega) &=\|\delta\omega+\partial\omega\|^{2}= \sum_{\Sigma}m\vert\delta\omega+\partial\omega\vert^2.
\end{align*}
Throughout the rest of this paper, $\bullet$  will stand for either $+,-$ or $H$, i.e.,
\begin{equation*}
	\bullet\in \{+,-,H\}.
\end{equation*} We define the spaces of functions of finite energy as $$\mathcal{D}^\bullet =\mathcal{D}^\bullet (\Sigma,m)= \lbrace \omega\in \F\mid\Q^\bullet(\omega)<\infty\rbrace,$$
which is a vector space due to the Cauchy-Schwarz inequality.

From this point on, we will always assume that the weight $m$ is \textit{locally summable}, i.e., for all $\tau\in\Sigma$, we have
$$\sum_{\sigma\succ\tau}m(\sigma)<\infty.$$

This has several important consequences that we summarize as follows.

\begin{lemma}[Locally summable]\label{sammellemma2}
    The weight $m$ is locally summable iff $\F^c\subseteq\mathcal{D}^+.$
    Moreover, if $m$ is locally summable, then the following inclusions hold:
    \begin{itemize}
        \item[(a)] $\F^c\subseteq\mathcal{D}^\bullet$.
        \item[(b)] $\ell^2\subseteq \F^1$.
        \item[(c)] $\F^c\subseteq D(\partial\delta)$.
        \item[(d)] $\delta \F^c\subseteq \F^2$.
    \end{itemize} 
\end{lemma}
\begin{proof}
    Let $\tau\in\Sigma$ be fixed. Then $$\Q^{+}(1_{\tau}) =\sum_{\sigma\in \Sigma}m(\sigma)\vert\delta1_{\tau}(\sigma)\vert^2 = \sum_{\sigma\succ \tau}m(\sigma),$$
    which is finite for all choices of $\tau$ iff $m$ is locally summable. This shows the desired equivalence as well as $(a)$ for $\bullet = +.$ The rest of $(a)$ follows from Lemma \ref{sammellemma1} as $\partial \F^c\subseteq \F^c\subseteq\ell^{2}$ so $\F^{c}\subseteq\mathcal{D}^-$ and, therefore, $\F^{c}\subseteq \mathcal{D}^{+}   \cap\mathcal{D}^{-}\subseteq \mathcal{D}^{H}$ which holds by the triangle inequality.\\
    To see $(b)$, let $\omega\in\ell^2.$ Then, for $\tau\in\Sigma$, we see with the Cauchy-Schwarz inequality
    $$\sum_{\sigma\succ \tau}m(\sigma)\vert\omega(\sigma)\vert\leq \Vert\omega\Vert\left(\sum_{\sigma\succ \tau}m(\sigma)\right)^{1\slash 2},$$
    which is finite by the local summability assumption. Hence, $\omega\in \F^{1}$.\\
  Statement  $(c)$ follows from $(a)$ and $(b)$ because $(a)$ implies that $\delta \F^c\subseteq \ell^2.$\\
    To see $(d)$ we fix $\rho,\tau_0\in\Sigma.$ Then
    \begin{multline*}
   \sum_{\sigma\succ \tau\succ\rho}m(\sigma)\vert\delta 1_{\tau_0}(\sigma)\vert = \sum_{\sigma\succ \tau\succ\rho} m(\sigma)1_{\tau_0\prec\sigma} = \sum_{\sigma\succ\tau_0}m(\sigma)\sum_{\tau}1_{\sigma\succ \tau\succ\rho}\\
    \leq 2\sum_{\sigma\succ\tau_0}m(\sigma)<\infty.
     \end{multline*}
    Here, 
     that for fixed $\rho,\sigma\in\Sigma$ there are at most two
     $\tau\in\Sigma$ such that $\sigma\succ \tau\succ\rho.$
    \end{proof}
  
  We introduce an even stronger condition. We say $m$ is \emph{strongly locally summable} if it is locally summable and for all $\rho\in \Sigma$      $$\sum_{\tau\succ \rho}\sum_{\sigma\succ\tau}m(\sigma)=\sum_{\sigma\succ\tau\succ \rho}m(\sigma)<\infty.$$ 
  \begin{lemma}[Strongly locally summable]\label{sammellemma3} If $m$ is strongly locally summable, then
  	$$\ell^2\subseteq\F^2.$$
  	In particular, $\partial\partial=0$ on $\ell^{2}$.
  \end{lemma}
  \begin{proof}
  	Let $\rho\in\Sigma$ and $\omega\in\ell^2.$ By the Cauchy-Schwarz inequality and the fact that for each $\sigma$ there are at most two $\tau$ such that $\rho\prec\tau \prec \sigma$, we obtain $\omega\in F^{2}$ since
  \begin{equation*}
  		\left(\sum_{\sigma\succ \tau\succ\rho}m(\sigma)\vert\omega(\sigma)\vert\right)^2=\left(\sum_{\sigma\in \Sigma}m(\sigma)\vert\omega(\sigma)\vert\sum_{\tau\prec\sigma}1_{\rho\prec \tau}\right)^2\leq   2\Vert\omega\Vert^2\sum_{\sigma\succ \tau\succ\rho} m(\sigma),
  \end{equation*}
  which is finite due to strong local sumability. The ``in particular'' statement follows from Lemma~\ref{pp0}.  \end{proof}

\subsection{Dirichlet and Neumann Laplacians}\label{Dir&Neum}
In this section, we introduce self-adjoint operators associated to two natural 
restrictions of $\mathcal{Q}^{\bullet}$, $\bullet\in\{\pm, H\}$.

Let us first  define 
$$Q^\bullet_N = \Q^\bullet\mid_{D(Q^\bullet_N)}, \qquad D(Q^\bullet_N) = \mathcal{D}^\bullet\cap \ell^2.
	$$
By Fatou's lemma $\Q^{\bullet}$ restricted to $\mathcal{D}^{\circ}\cap \ell^{2}$ is lower semi continuous and, therefore,  closed, see e.g. \cite[Theorem B.9]{KLW}.  By general theory, see e.g. \cite[Theorem 5.37]{W}, the positive closed quadratic form $Q^{\bullet}_{N}$ gives rise to a
 unique positive self-adjoint operator $\Delta^\bullet_N$ such that for all $\omega\in D(Q^\bullet_N)$ and $\eta\in D(\Delta^\bullet_N)$ one has
$$Q^\bullet_N(\omega,\eta) = \langle \omega,\Delta^\bullet_N\eta\rangle.$$ 
We respectively call the operators $\Delta^+_N,\Delta^-_N,\Delta^H_N$ the \textit{discrete up-, down-} and \textit{Hodge-Laplacian} with \textit{Neumann boundary conditions.} The name arises as $D(Q^{\bullet}_{N})$ corresponds  to the Sobolev spaces $W^{1,2}$ in the continuum.
Furthermore, note that by the triangle inequality $$D(Q^+_N)\cap D(Q^-_N)\subseteq D(Q^H_N).$$

Since $Q^{\bullet}_{N}$ is closed, its restriction to $F^{c}$ is closable in $\ell^{2}$ and we define 
 $$Q^{\bullet}_{D} = Q^\bullet_{N}\mid_{D(Q^\bullet_{D})},\quad D(Q^{\bullet}_{D}) = \overline{\F^c}^{\Vert\cdot\Vert_{\Q^{\bullet}}},$$
where $\Vert\cdot\Vert_{\Q^\bullet}^2 = \Q^\bullet(\cdot)+\Vert\cdot\Vert^2_2.$ These are positive, closed restrictions of $\Q^\bullet$ and give rise to unique positive self-adjoint  operators $\Delta^{\bullet}_{D}$ such that for all $\omega\in D(Q^\bullet_{D})$ and $\eta\in D(\Delta^\bullet_{D})$ one has
$$Q^\bullet_{D}(\omega,\eta) = \langle \omega,\Delta^\bullet_{D}\eta\rangle.$$
Note that by Corollary \ref{adj2} we have 
 $$Q^H_{D}  =(Q^+_{D}+Q^-_{D})\mid_{D(Q^H_{D})}\quad\mbox{ and } \quad D(Q^H_{D}) \subseteq D(Q^+_{D})\cap D(Q^-_{D}).$$ We respectively call the operators $\Delta^+_{D},\Delta^-_{D},\Delta^H_{D}$ the \textit{discrete up-, down-} and \textit{Hodge-Laplacian} with \textit{Dirichlet boundary conditions.} The name arises as one may consider the form domains $D(Q^{\bullet}_{D})$ as analogues of the Sobolev spaces $W^{1,2}_{0} $ in the continuum.
 
\begin{definition}[Laplacian]\label{def:Laplacian} A \emph{Laplacian} $\Delta^{\bullet}$, $\bullet\in\{\pm,H\}$, associated to a simplicial complex $(\Sigma,m)$ is  the unique positive self-adjoint operator associated to a closed quadratic form $Q^{\bullet}_{D}\subseteq Q^{\bullet}\subseteq Q^{\bullet}_{N}$.
\end{definition}
 Next, we determine the action of these Laplacians.

\begin{theorem}[Action of  Laplacians]\label{action} Let a Laplacian $\Delta^{\bullet}$  be given,  $\bullet\in\{\pm,H\}$. Then, $D(\Delta^{+})\subseteq D(\partial\delta)$,  $D(\Delta^{-})\subseteq D(\delta\partial)$, and  $D(\Delta^{H})\subseteq D((\delta+\partial)^{2})$, and
	\begin{align*}
		\Delta^+ &=\partial\delta&&\hspace{-3cm}\mbox{ on }D(\Delta^{+}),\\
		\Delta^- &=\delta\partial&&\hspace{-3cm} \mbox{ on }D(\Delta^{-}),\\
		\Delta^H &=(\delta+\partial)^{2}&&\hspace{-3cm} \mbox{ on }D(\Delta^{H}).
	\end{align*}
	If $\Delta^{H}=\Delta^{H}_{D} $ or $m$ is strongly locally summable, then  $D(\Delta^{H})\subseteq D(\partial\delta+\delta\partial)$
	and
	\begin{align*}
			\Delta^H &=\partial\delta+\delta\partial&&\hspace{-3cm} \mbox{ on }D(\Delta^{H}).
	\end{align*}
\end{theorem}
We start the proof with a lemma which is interesting in its own right. 
Observe that in general, we only have $\Delta^{H}=(\delta+\partial)^{2}$ but not necessarily equality with $\partial\delta+\delta\partial$. Indeed, in Example~\ref{example}, we discussed that $\partial\partial \omega_{0}\ne0$ for a certain $\omega_{0}\in D(\partial\partial)$ while $\partial\partial \omega =0$ for $\omega \in F^{2}$ by Lemma~\ref{pp0}.  We extend this observation to obtain the formula $\Delta^{H}_{D}=\partial\delta+\delta\partial.$
\begin{lemma}\label{pp01}
	Let $\omega\in D(Q^-_{D})$ and $\eta\in \mathcal{D}^+.$ Then $$\partial\partial\omega = 0\qquad \mbox{ and }\qquad \langle\partial\omega,\delta\eta\rangle = 0.$$
\end{lemma}
\begin{proof}
	Observe that $\partial\omega\in\ell^2\subseteq F^1 = D(\partial)$ and that there is a sequence $(\omega_n)_n$ in $F^c$ such that $\partial\omega_n\rightarrow\partial\omega$ in $\ell^2.$ Moreover, for $\rho\in\Sigma$ and according to Stokes' theorem, Theorem \ref{adj1}, we have $$m(\rho)\partial\partial\omega(\rho) = \langle \partial\omega,\delta1_\rho\rangle.$$
	Since $\partial\omega$ is the limit of $(\partial\omega_n)$ in $\ell^2$,  Corollary \ref{adj2}  gives both statements.
\end{proof}

\begin{proof}[Proof of Theorem~\ref{action}]
First, let $\omega\in D(\Delta^+)\subseteq \mathcal{D}^{+}$. Hence, $\delta\omega\in \ell^{2}\subseteq F^{1}=D(\partial)$ by Lemma~\ref{sammellemma2}~$(b)$, so we get $\omega\in D(\partial\delta).$ Now, the statement follows from Green's formula, Theorem~\ref{green1}~$(b)$,  with $\varphi = 1_\tau/m$  
$$\Delta^+\omega(\tau) = \langle\varphi,\Delta^+\omega\rangle = Q^+(\varphi,\omega) = \sum_{\Sigma}m \varphi \partial\delta\omega  = \partial\delta\omega(\tau).$$
To see  $\Delta^-\omega= \delta\partial\omega$, recall that $D(\Delta^-)\subseteq\ell^2\subseteq \F^1=D(\delta\partial).$ Then the statement follows from Green's formula, Theorem~\ref{green1}~$(a)$,  similar to $\Delta^{+}$.\\
For $\Delta^H$, let $\omega\in D(\Delta^{H})$. Observe first that $(\delta+\partial)\omega\in\ell^2$ because $\omega\in\mathcal{D}^H.$  Secondly, by local summability, Lemma~\ref{sammellemma2}~$(b)$, we have $\ell^{2}\subseteq F^{1}=D(\partial)$. As $D(\delta)=F$, we therefore have $(\delta+\partial)\omega\in D(\delta+\partial)$, i.e., $\omega\in D((\delta+\partial)^{2})$. Now, let additionally $\varphi = 1_{\tau}/m,$ for $\tau\in \Sigma$. Then, similar as for $\Delta^+,$
$$\Delta^H\omega(\tau) = Q^H(\varphi,\omega) = \sum_\Sigma m(\partial+\delta)\varphi(\partial+\delta)\omega = \sum_\Sigma m\varphi(\partial+\delta)^2\omega = (\partial+\delta)^2\omega(\tau),$$
where the second to last equality follows by applying Stokes' theorem, Theorem~\ref{adj1}, twice (recall that $(\partial+\delta)\omega\in F^1).$\\
Finally, if $\Delta^H = \Delta^H_D$ or if $\Sigma$ is strongly locally summable, then by Corollary \ref{adj2}
$$\sum_\Sigma m\partial\omega\delta\varphi = 0 = \sum_\Sigma m\delta\omega\partial\varphi.$$
Thus, 
employing Green's formula, Theorem \ref{green1} \emph{(a)} and \emph{(b)}, gives
$$\Delta^H\omega(\tau) = \sum_\Sigma m\delta\omega\delta\varphi+\sum_\Sigma m\partial\omega\partial\varphi = \sum_\Sigma m\varphi(\delta\partial+\partial\delta)\omega = (\delta\partial+\partial\delta)\omega(\tau).\hfill \qedhere$$
\end{proof}

Below we discuss that Example~\ref{example} demonstrates that $\Delta^{H}_{D}\neq \Delta^{H}_{N}$ in general. Later in Section~\ref{FormeqEssSa}, we give criteria when the operators indeed coincide.

\begin{example}[Example \ref{example} continued]\label{example_cont}
	Reconsidering Example \ref{example} we find that $\omega_{0}\in\ell^2\cap\mathcal{D}^+\cap\mathcal{D}^-\subseteq D(Q^H_N).$ On the other hand, 
	$$\vert\langle\delta\omega_{0},\partial\omega_{0}\rangle\vert = m(\lbrace 0,1\rbrace) = 1\neq 0.$$ However, this implies that $\omega_{0}\not\in D(Q^H_{D})\subseteq D(Q^-_{D})$ as this would contradict Lemma~\ref{pp01}. As a consequence, $Q^H_{D}\neq Q^H_N$ and in particular $\Delta^H_D\neq\Delta^H_{N}.$
\end{example}

Let us define $$ \delta_D=\delta\mid_{D(Q^+)}, \quad \partial_D=\partial\mid_{D(Q^-)} \quad\mbox{ and }\quad  (\partial+\delta)_D=(\delta+\partial)\mid_{D(Q^H)} $$ 
 $$\delta_N=\delta\mid_{D(Q^+_N)} ,\quad \partial_N=\partial\mid_{D(Q^-_N)} \quad \mbox{ and }\quad(\partial+\delta)_N=(\delta+\partial)\mid_{D(Q^H_N)} .$$\\
By \cite[Theorem 5.1]{W} these are closed operators and by \cite[Theorem~5.39]{W} the compositions with their respective adjoints are self-adjoint operators.

The next lemma shows in which precise sense  $\delta$ and $\partial$ are adjoints of each other.
\begin{lemma}[Adjoints of $\delta$ and $\partial$]\label{d*}
    We have 
   $$\delta^*_D = \partial_N,\qquad\partial^*_D = \delta_N\qquad\mbox{and}\qquad(\delta+\partial)^*_D = (\delta+\partial)_N,$$
    $$\delta_D = \partial^*_N,\qquad\partial_D = \delta^*_N\qquad\mbox{and}\qquad(\delta+\partial)_D = (\delta+\partial)^*_N.$$
\end{lemma}
\begin{proof}
    First, let $\omega\in D(\delta^*_D)$, i.e., for every $\eta\in  D(\delta_D)$, we have
    $$\langle\omega,\delta\eta\rangle = \langle\delta^*_D\omega,\eta\rangle.$$
    Choosing $\eta = 1_\rho\slash m$ and using Stokes' theorem, Theorem~\ref{adj1}, shows $\delta^*_D\omega = \partial\omega.$ In particular, $\partial\omega\in\ell^2$ and consequently $\omega\in\mathcal{D}^-\cap\ell^2.$ Hence, $\delta^*_D \subseteq \partial_N.$\\
    Now, let $\omega\in\mathcal{D}^-\cap\ell^2=D(Q_{N}^{-})=D(\partial_{N})$ and $\eta\in D(Q^+_D).$ We can choose a sequence $(\eta_n)$ in $\F^c$ such that $\eta_n\rightarrow\eta$ and $\delta\eta_n\rightarrow\delta\eta$  in $\ell^2.$ As $\partial\omega,\omega\in\ell^2$ we deduce from Stokes' theorem, Theorem \ref{adj1}, once again
    $$\langle\omega,\delta\eta\rangle = \lim_{n\to\infty}\langle\omega,\delta\eta_n\rangle = \lim_{n\to\infty}\langle\partial\omega,\eta_n\rangle = \langle\partial\omega,\eta\rangle.$$
    In particular, $\omega\in D(\delta^*_D)$ and, therefore, $  \partial_N\subseteq \delta^*_D$ which gives $\delta^*_D = \partial_N.$ The equality $\delta_D=\delta_{D}^{**} = \partial_N^*$ follows by  \cite[Theorem~5.3~(b)]{W} since $\delta_{D}$ is closed.
    The other two cases follow along the same line of arguments.
\end{proof}

\begin{corollary}\label{LNLD}
    We have
    $$\Delta^+_{D} = \partial_N\delta_D,\quad\Delta^-_{D} = \delta_N\partial_D\quad\text{and}\quad\Delta^H_{D}= (\delta+\partial)_N(\delta+\partial)_D,$$
    $$\Delta^+_N = \partial_D\delta_N,\quad\Delta^-_N = \delta_D\partial_N\quad\text{and}\quad\Delta^H_N = (\delta+\partial)_D(\delta+\partial)_N.$$
\end{corollary}
\begin{proof}
A straight forward calculation that uses the definition of adjoint operators and Green's formulas, Theorem \ref{green1}, shows that 
$$\Delta^+_{D} = \delta_D^*\delta_D,\quad\Delta^-_{D} = \partial_D^*\partial_D\quad\text{and}\quad\Delta^H_{D} = (\delta+\partial)^*_D(\delta+\partial)_D$$
as well as 
$$\Delta^+_N = \delta^*_N\delta_N,\quad\Delta^-_N = \partial_N^*\partial_N\quad\text{and}\quad\Delta^H_N = (\delta+\partial)^*_N(\delta+\partial)_N.$$
    Thus, the statements  are a direct consequence of Lemma~\ref{d*}. 
\end{proof}

\subsection{The Hilbert Complex Perspective} \label{Hilbert}

In this section we will see how the operators introduced in Section \ref{quadform} can be described in the language of Hilbert complexes. For a more detailed introduction to Hilbert complexes, we refer to \cite{BL} (see also \cite{HolstStern}) whose setting we will closely follow here. We denote the range of an operator $A$ by $R(A).$

\begin{definition}[Hilbert complex]
Given a sequence of mutually orthogonal Hilbert spaces $H_k$ and closed, densely defined operators $d_{k}$ on $H_{k} $ to $H_{k+1}$ such that $R(d_{k})\subseteq D(d_{k+1})$ and $$d_{k+1}d_k = 0$$ 
for $k\in\mathbb{Z}$,
we obtain a \textit{Hilbert complex} in form of a chain complex (in the sense of homological algebra)
$$\ldots\xrightarrow{d_{k-2}} D(d_{k-1})\xrightarrow{d_{k-1}}D(d_k)\xrightarrow{d_k}D(d_{k+1})\xrightarrow{d_{k+1}}\ldots.$$
\end{definition}

It follows easily from the definition of a Hilbert complex that the adjoints of $d_k$ form a Hilbert complex as well, called the \textit{dual complex}, i.e., $R(d_{k}^{*})\subseteq D(d_{k+1}^{*})$ and $d_{k}^{*}d_{k+1}^{*}=0$,
$$\ldots\xleftarrow{d^*_{k-1}} D(d^*_{k-1})\xleftarrow{d^*_{k}}D(d^*_k)\xleftarrow{d^*_{k+1}}D(d^*_{k+1})\xleftarrow{d^*_{k+2}}\ldots.$$
Given a general Hilbert complex associated to operators $d_k$ and $d = \bigoplus_kd_k$ on $H=\bigoplus_{k}H_{k}$, one defines a Laplacian $\Delta$ as $$\Delta = d^*d+dd^*,$$ which by von Neumann's theorem \cite[Theorem~5.39]{W} is always self-adjoint as $\Delta = (d+d^*)^2$ and $d+d^*$ is self-adjoint. The fact that $d+d^*$ is densely defined is a consequence of $d_{k+1}d_k = 0.$ The operator $d+d^*$ is also known as Gauss-Bonnet-Operator (\cite{ATH,parra2017spectral}). The quadratic form associated to $\Delta$ on $H$ is given by $D(d+d^{*})\to \mathbb{R}$, $f\mapsto \|(d+d^{*})f\|^{2}$. An important tool within this theory is the weak Hodge-decomposition, which is shown in
  \cite[Lemma~2.1. and Lemma~2.2.]{BL}
$$H = \ker\Delta \oplus\overline{R(d)}\oplus\overline{R(d^{*})}.$$ For the convenience of the reader we will provide a proof of this result in context of simplicial complexes below.

First, we show how our operators fit in this framework. To associate a Hilbert complex to a simplicial complex, we set $H_{-1} = \lbrace0\rbrace$ or $H_{-1} = \ell^2(\Sigma_{-1},m)\cong\mathbb{R},$ depending on whether $\Sigma=\hat{\Sigma}\setminus \{\varnothing\}$ or $\hat\Sigma = \Sigma$, $H_k = \ell^2(\Sigma_k,m)$, $k> -1$ and $H_k = \lbrace 0\rbrace$, $k<-1,$ as well as $d_k = \delta_{N,k}$, where $\delta_{N,k}$ is the restriction of $\delta_N$ to $\ell^2(\Sigma_k,m)\cap D(Q^+_N).$ Those are closed and densely defined by the same arguments as $\delta_N$ is closed and densely defined.
The defining property of $\delta$, i.e., $\delta\delta=0$ on $F$ gives $$R(\delta_{N})  \subseteq \mathcal{D}^{+}\cap \ell^{2}= D(\delta_{N})\quad \mbox{and}\quad \delta_{N}\delta_{N}=0, $$ which then transfers to $\delta_{N,k}$. Hence, we obtain a Hilbert complex
$$0\rightarrow D(\delta_{N,-1})\xrightarrow{\delta_{N,-1}}D(\delta_{N,0})\xrightarrow{\delta_{N,0}}D(\delta_{N,1})\xrightarrow{\delta_{N,1}}\ldots.$$ 
Furthermore, in view of Lemma \ref{d*}
we have $\partial_{D}=\delta_{N}^{*}$ which gives rise to the dual complex by defining $\partial_{D,k}$ as the restriction of $\partial_D$ to $D(\partial_{D,k})=\ell^2(\Sigma_k,m)\cap D(Q^-_{D})$ for $k\in\mathbb{N}$
$$0\leftarrow H_{-1}\xleftarrow{\partial_{D,0}}D(\partial_{D,0})\xleftarrow{\partial_{D,1}}D(\partial_{D,1})\xleftarrow{\partial_{D,2}}\ldots.$$
Note, that $H_{-1} = D(\delta_{N,-1}) = R(\partial_{D,0}).$
The self-adjoint Laplacian $\Delta=\Delta_{\Sigma}$ associated to the Hilbert complex is then given by
$$\Delta_{\Sigma} = \delta_N\partial_D+\partial_D\delta_N = (\delta_N+\partial_D)^2 = \Delta^+_N+\Delta^-_{D}.$$
The quadratic form $Q_{\Sigma}$ associated to $\Delta_{\Sigma}$ is given
as 
\begin{align*}
	Q_{\Sigma}(\omega)=&\|(\delta_{N}+\partial_{D})\omega\|^{2}=\|\delta_{N}\omega\|^{2}+\|\partial_{D}\omega\|^{2},\\ D(Q_{\Sigma})=&D(\delta_{N}+\partial_{D})=D(Q_{N}^{+})\cap D(Q_{D}^{-}).
\end{align*}
Hence,
\begin{equation*}
	Q_{D}^{H}\subseteq Q_{\Sigma}\subseteq Q_{N}^{H}.
\end{equation*}

The following result then relates harmonic functions with respect to $\Delta$ to the reduced $\ell^2$-cohomology groups $H^k_{(2)}(\Sigma),$ which are defined as 
$$H^k_{(2)}(\Sigma) = \ker\delta_{N,k}\slash \overline{R(\delta_{N,k-1})}.$$

\begin{theorem}[Weak Hodge-Decomposition]\label{kodaira}
	One has the orthogonal decomposition
	$$\ell^2 = \ker\Delta_\Sigma\oplus\overline{R(\delta_N)}\oplus\overline{R(\partial_D)}.$$
	In particular, for the restriction $\Delta_{\Sigma,k}$ of $\Delta_{\Sigma}$ to $\ell^{2}(\Sigma_{k},m)$, $k\ge0$,
	$$H^k_{(2)}(\Sigma)\cong\ker\Delta_{\Sigma,k}.$$
\end{theorem}
\begin{proof}
	First, we argue that $$\ker\Delta_\Sigma =\ker(\delta_N+\partial_D)= \ker\delta_N\cap\ker\partial_D.$$ The first equality follows since $\Delta_\Sigma = (\delta_N+\partial_D)^2 = (\delta_N+\delta_N^*)^2$ by Lemma~\ref{d*} and
	since $\ker T^{*}T=\ker T$ for a self-adjoint operator $ T$, see e.g. \cite[Theorem~5.39]{W}.  The second equality uses $\delta_N\delta_N = 0,$ which readily gives $\langle \delta_{N}\omega,\partial_{D}\eta\rangle=0$ again by Lemma~\ref{d*}. Thus, $R(\delta_{N})\perp R(\partial_{D}),$ which implies the inclusion $\subseteq $ in the second equality, while the other inclusion is trivial. \\
    As $\delta_N$ is closed and $\delta_{N}^{*}=\partial_{D}$, we have $\ker (\delta_{N})^{\perp}=\overline{R(\delta_{N}^{*})}=\overline{R(\partial_{D})}$ by \cite[Theorem~4.13]{W} and, therefore,    
     $$\ell^2 = \ker\delta_N\oplus(\ker\delta_N)^\perp = \ker\delta_N\oplus\overline{R(\partial_D)}.$$ 
     Moreover, as $\delta_N\delta_N = 0$, we have $R(\delta_{N})\subseteq \ker\delta_{N}$ and since $\ker\delta_{N}$ is closed. Thus,
      $$\ker\delta_N = \overline{R(\delta_N)}\oplus (\ker\delta_N\cap\overline{R(\delta_N)}^\perp) = \overline{R(\delta_N)}\oplus(\ker\delta_N\cap\ker\partial_D),$$
      where we used again $ \overline{R(\delta_{N})}^{\perp}=\ker(\delta_{N}^{*})=\ker{(\partial_{D})}$ by \cite[Theorem~4.13]{W} and Lemma~\ref{d*}.
    All together, this yields the desired decomposition.\\
    Restricting to $\ell^2(\Sigma_k,m)$ and taking the quotient over $\overline{R(\delta_N)}$ yields the result about the $\ell^2$-cohomology.
    \end{proof}

In Section \ref{FormeqEssSa}, we provide criteria for $Q^{H}_{D}=Q^{H}_{N},$ which is equivalent to $\Delta_{D}^{H}=\Delta_{\Sigma}=\Delta_{N}^{H}$ since $Q_{D}^{H}\subseteq Q_{\Sigma}\subseteq Q_{N}^{H}$. In this case,  Theorem~\ref{kodaira} applies to $\Delta_{D}^{H}$, $\Delta_{N}^{H}$ as well.\\

It is natural to ask whether we obtain Hilbert complexes by considering $\delta_D$ or $\partial_N$ instead. 
Again, Lemma~\ref{d*} gives $\delta_{D}^{*}=\partial_{N}$. So by the discussion above,  if one of the two gives rise to a Hilbert complex, then so does the other.
However, a counterexample for $\partial_{N}\partial_{N}=0$
is given in Example~\ref{example} and its continuation in Example~\ref{example_cont}, which shows that for $\partial_N$ this is not the case in general. On the other hand, we always have $\delta_{D}\delta_{D}=0,$ which, however, may not be densely defined. Thus,  $\delta_{D}$ and $\partial_{N}$ yield Hilbert complexes if and only if
\begin{equation*}
	R(\delta_{D})\subseteq D(\delta_{D}).
\end{equation*}
In this case, we obtain a self-adjoint Laplacian 
$$\Delta'_{\Sigma} = \delta_D\partial_N+\partial_N\delta_D = (\delta_D+\partial_N)^2 = \Delta^+_{D}+\Delta^-_N$$
whose quadratic form $Q'_{\Sigma}$ also satisfies $$Q_{D}^{H}\subseteq Q'_{\Sigma}\subseteq Q^{H}_{N}.$$

Obviously, this is trivially satisfied if $\Delta_{D}^{+}=\Delta_{N}^{+}$ which simply means $Q^+ _{D}= Q^+_N$ and  $\delta_D = \delta_N.$ 

Another criterion is given by strong local summability of the weight $m$.
\begin{theorem}\label{kodaira2}
	If the weight $m$ is strongly locally summable, then $\delta_{D}$ and $\partial_{N}$ give rise to Hilbert complexes and 
		$$\ell^2 = \ker\Delta'_{\Sigma}\oplus\overline{R(\delta_D)}\oplus\overline{R(\partial_N)}$$
	and
	$$\ker\delta_{D,k}\slash \overline{R(\delta_{D,k-1})}\cong \ker\Delta_{\Sigma,k}'.$$
\end{theorem}
\begin{proof}
	By Lemma~\ref{sammellemma3}  strong local summability implies  $\partial\partial =0$ on $\ell^{2}$. 
	Consequently, $R(\partial_N)\subseteq D(\partial_N)$ and $\partial_N\partial_N = 0$. By Lemma~\ref{d*}, $\delta_{D}^{*}=\partial_{N}$. Thus, $\partial_{N}$ and $\delta_{D}$ give rise to Hilbert complexes. The rest of the proof carries over verbatim from the proof of Theorem~\ref{kodaira}.
\end{proof}

\subsection{The Signed Schrödinger Operator Perspective}\label{DiscSch}

Our goal in this section is to relate the Laplacians introduced above to discrete signed Schrö\-dinger operators. This will enable us to prove statements about Laplacians by exploiting the theory that has already been developed there.

Given a locally summable weight $m$ on a simplicial complex $\Sigma$ and $\bullet\in\{\pm,H\}$, we define $b^{\bullet}:\Sigma\times\Sigma\to[0,\infty)$ for $\tau,\tau'\in \Sigma$ by
   \begin{align*}
	b^+(\tau,\tau') &= m(\tau\cup\tau')1_{\tau,\tau'\prec\tau\cup\tau'},\\
	b^-(\tau,\tau') &= \frac{m(\tau)m(\tau')}{m(\tau\cap\tau')}1_{\tau\cap\tau'\prec\tau,\tau'},\\
	b^H(\tau,\tau') & = \vert b^-(\tau,\tau')-b^+(\tau,\tau')\vert.
\end{align*}
Then $b^{\bullet}$ is symmetric, zero on the diagonal and satisfies for all $\tau\in\Sigma$
\begin{equation*}
	\sum_{\tau'\in \Sigma}b^{\bullet}(\tau,\tau')<\infty
\end{equation*}
by local summability. Observe that if $b^{\bullet}(\tau,\tau')>0$, then $\dim(\tau)=\dim(\tau')$.
Furthermore, let $o^{\bullet}:\Sigma\times\Sigma\to\{-1,0,1\}$ be  given for $\tau,\tau'\in \Sigma$ by
\begin{align*}
	o^+(\tau,\tau') &= \theta(\tau,\tau\cup\tau')\theta(\tau',\tau\cup\tau') 1_{\tau,\tau'\prec\tau\cup\tau'}\\
	o^-(\tau,\tau') &= \theta(\tau\cap\tau',\tau)\theta(\tau\cap\tau',\tau')1_{\tau\cap\tau'\prec\tau,\tau'},\\
	o^H(\tau,\tau') &= o^-(\tau,\tau')\mathrm{sgn}(b^-(\tau,\tau')-b^+(\tau,\tau')).
\end{align*}
Clearly, $o^{\bullet}$ is symmetric and is zero on the diagonal as well. 
Finally, let $c^{\bullet}:\Sigma\to\mathbb{R}$ be given for $\tau\in \Sigma$ by
\begin{align*}
	c^+(\tau) &= -\mathrm{dim}(\tau)\sum_{\tau\prec\sigma}m(\sigma),\\
	c^-(\tau) &= \sum_{\rho\prec\tau}\frac{m(\tau)}{m(\rho)}\Big(m(\tau)-\sum_{\rho\prec\tau'\neq\tau}m(\tau')\Big),\\
	c^H(\tau) &= \sum_{\rho\prec\tau}\frac{m^2(\tau)}{m(\rho)}+\sum_{\tau\prec\sigma}m(\sigma)-\sum_{\rho\prec\tau}\sum_{\rho\prec\tau'\neq\tau}\Big\vert\frac{m(\tau)m(\tau')}{m(\rho)}-\sum_{\tau,\tau'\prec\sigma}m(\sigma)\Big\vert.
\end{align*}

These quantities give rise to signed Schrödinger operators, cf. \cite{GKS}. We define the \emph{formal signed Schrödinger operator} $\mathcal{H}^\bullet$ 	acting as
$$\mathcal{H}^\bullet\omega(\tau) = \frac{1}{m(\tau)}\sum_{\tau' \in \Sigma}b^\bullet(\tau,\tau')(\omega(\tau)-o^\bullet(\tau,\tau')\omega(\tau'))+\frac{c^\bullet(\tau)}{m(\tau)}\omega(\tau)$$
on  $$\F^\bullet = \F^\bullet(\Sigma) =\lbrace \eta\in F\mid\sum_{\tau'\in \Sigma}b^\bullet(\tau,\tau')\vert\omega(\tau')\vert<\infty~\text{for all}~\tau\in\Sigma\rbrace.$$
Because $b^{H}\leq b^{+}+b^{-}$ we see that $$F^{c}\subseteq F^{\bullet}\quad\mbox{ and }\quad F^{-}\cap F^{+}\subseteq F^{H}.$$ 
It is easy to show that $F^+$ coincides with the respective set in Theorem \ref{green1} (c), hence the same name.
The term $c^H\slash m$ is known in the literature as \textit{Forman curvature} or \textit{Forman-Ricci curvature} (\cite{F,JM}). In a sense, the Schrödinger operator representation for $\Delta^H$ resulting from Theorem \ref{schröd} below can be interpreted as Weitzenböck identity (\cite[Proposition A3.]{Chavel}) in which $c^H$ plays the role of the Ricci curvature, see for example \cite{F} and \cite{JM}. Furthermore, $o^{\circ}$ can be viewed as a magnetic potential.

The aim now is to relate $\mathcal{H}^{+},\mathcal{H}^{-}$ and $\mathcal{H}^{H}$ to $ \partial\delta$, $\delta\partial$ and $ \partial\delta+\delta\partial$.

\begin{theorem}[Action of $\mathcal{H}^{\circ}$]\label{schröd2}
	One has $F^{+}\subseteq  D(\partial\delta)$,  $F^{-} = D(\delta\partial)$, $\F^{+}\cap F^{-}\subseteq  D(\partial\delta+\delta\partial)$ and
	\begin{align*}
		\mathcal{H}^{+} &=\partial\delta&&\hspace{-2cm}\mbox{ on }\F^{+},\\
		\mathcal{H}^{-}&=\delta\partial&&\hspace{-2cm} \mbox{ on }F^{-}= D(\delta\partial),\\
		\mathcal{H}^{H}&= \partial\delta+\delta\partial&&\hspace{-2cm} \mbox{ on }\F^{+}\cap \F^{-}.
	\end{align*}
	If the simplicial complex is locally finite, then equality holds on $\F$. Furthermore, for $\eta\in \F^{c}$
	\begin{align*}
		\sum_{\Sigma}m (\mathcal{H}^{+}\omega)\eta &=	\sum_{\Sigma}m(\delta\omega)(\delta\eta),&&\omega\in \F^{+},\\
		\sum_{\Sigma}m(\mathcal{H}^{-}\omega)\eta&=	\sum_{\Sigma}m(\partial\omega)(\partial\eta),&& \omega\in F^{-}= D(\delta\partial),\\
		\sum_{\Sigma}m	(\mathcal{H}^{H}\omega)\eta&= 	\sum_{\Sigma}m\delta\omega\delta\eta+\sum_\Sigma m\partial\omega\partial\eta,&& \omega\in \F^{+}\cap \F^{-}.
	\end{align*}
\end{theorem}

The first step in the proof is achieved by the following representation of the quadratic form $Q^{\bullet}_{D}$.

\begin{lemma}\label{schröd}
  For $\bullet\in\{\pm,H\}$ and $\omega\in D(Q^{\circ}_{D})$,
    $$Q^\bullet_{D}(\omega) = \frac{1}{2}\sum_{\tau,\tau'\in\Sigma}b^\bullet(\tau,\tau')(\omega(\tau)-o^\bullet(\tau,\tau')\omega(\tau'))^2+\sum_{\tau\in\Sigma}c^\bullet(\tau)\omega^2(\tau).$$
\end{lemma}
\begin{proof}We denote $ Q^{\bullet}=Q^{\bullet}_{D}$ for the proof.
First, notice that for any fixed $\tau\in\Sigma$, one has with $C_{\tau}=(\mathrm{dim}(\tau)+1)$ $$\sum_{\tau'\neq\tau}\vert Q^+(1_\tau,1_{\tau'})\vert \leq \sum_{\tau'\neq\tau}\sum_{\tau,\tau'\prec\sigma}m(\sigma) = \sum_{\tau\prec\sigma}m(\sigma)\hspace{-.2cm}\sum_{\tau'\prec\sigma,\tau'\neq\tau}\hspace{-.2cm}1 = C_{\tau}\sum_{\tau\prec\sigma}m(\sigma)<\infty$$ by local summability. A similar calculation shows that $\sum_{\tau'\neq\tau}\vert Q^\bullet(1_\tau,1_{\tau'})\vert<\infty$ for all $\tau\in \Sigma$ and $\bullet \in\{-,H\}.$
Thus, by writing $\omega\in \F^c$ as $\omega = \sum_{\tau}\omega(\tau)1_\tau$ and rearranging terms, we may write $Q^\bullet(\omega)$ as  
\begin{multline*}
    Q^\bullet(\omega) = \frac{1}{2}\sum_{\tau\neq\tau'}\vert Q^\bullet(1_\tau,1_{\tau'})\vert\left(\omega(\tau)+\mathrm{sgn}(Q^\bullet(1_\tau,1_{\tau'}))\omega(\tau')\right)^2\\
    +\sum_{\tau\in\Sigma}\Big(Q^\bullet(1_\tau)-\sum_{\tau\neq\tau'}\vert Q^\bullet(1_\tau,1_{\tau'})\vert\Big)\omega^2(\tau).
\end{multline*}
    Therefore, we obtain the asserted representation with
    \begin{align*}
        b^\bullet(\tau,\tau') &= \vert Q^\bullet(1_\tau,1_{\tau'})\vert,\\
        o^\bullet(\tau,\tau') &= -\mathrm{sgn}(Q^\bullet(1_\tau,1_{\tau'})),\\
        c^\bullet(\tau) &= Q^\bullet(1_\tau)-\sum_{\tau\neq\tau'}\vert Q^\bullet(1_\tau,1_{\tau'})\vert
    \end{align*}
     for $\tau\neq\tau'$ and $b^{\bullet}=o^{\bullet}=0$ on the diagonal.
    To finish the proof, we  compute $b^{\bullet}$, $o^{\bullet}$ and $c^{\bullet}$ explicitly. The result for $\omega\in D(Q^{\bullet})$ follows then by density of $\F^{c}$ in $D(Q^{\bullet})$.
    
    We see that for $\tau\neq\tau'$,  
    $$Q^+(1_\tau,1_{\tau'}) = \sum_{\sigma\in \Sigma} m(\sigma)\delta 1_\tau(\sigma)\delta 1_{\tau'}(\sigma) = \sum_{\sigma\succ\tau,\tau'}  m(\sigma)\theta(\tau,\sigma)\theta(\tau',\sigma).$$
   Note that there is at most one $\sigma\succ\tau,\tau'$ which is then $\tau\cup\tau'$. This immediately yields the representations of $b^+$ and $o^+.$ The summability of $b^{+}$ follows directly from the local summability of $m$.    
   Likewise, we get $Q^+(1_\tau) = \sum_\sigma m(\sigma)1_{\tau\prec\sigma}.$ Consequently, as there are exactly $\dim (\tau )+2$ faces in every $\sigma\succ\tau$
    $$c^+(\tau) = \sum_{\tau\prec\sigma}m(\sigma)-\sum_{\tau'\neq\tau}\sum_{\tau,\tau'\prec\sigma}m(\sigma)
     = -\mathrm{dim}(\tau)\sum_{\tau\prec\sigma}m(\sigma).$$
    Analogously we see that for $\tau\neq\tau'$
    $$Q^-(1_\tau,1_{\tau'}) =\sum_{\rho\in \Sigma} m(\rho)\partial 1_\tau(\rho) \partial 1_{\tau'}(\rho) = \sum_{\rho\prec \tau,\tau'} \frac{1}{m(\rho)}m(\tau)\theta(\rho,\tau)m(\tau')\theta(\rho,\tau').$$
    There is at most one $\rho\prec\tau,\tau'$ which is then $\tau\cap{\tau}'$.
    This yields the representations for $b^-$ and $o^-.$  The summability of $b^{-}$ follows directly from the local summability of $m$.    
      Likewise, $Q^-(1_\tau) = \sum_{\rho\prec \tau}{m(\tau)^2}/{m(\rho)}$ and the representation of $c^-$ follows immediately.\\
    Finally, as $Q^H = Q^++Q^-$ on $F^c,$ we find that  
    $$Q^H(1_\tau,1_{\tau'}) =  o^-(\tau,\tau')b^-(\tau,\tau')+o^+(\tau,\tau')b^+(\tau,\tau').$$
    The representations of $b^H$ and $o^H$ follow by noting that $o^+(\tau,\tau') = -o^-(\tau,\tau')$ whenever $\tau$ and $\tau'$ share a common coface $\sigma,$ which is  a reformulation of $\delta\delta 1_\rho(\sigma) = 0$ (cf. the proof of Lemma~\ref{h}). The representation of $c^H$ also follows immediately.
\end{proof}

To relate the representation for the quadratic form to the operator $\mathcal{H}^{\bullet}$ we need a Green's formula.

\begin{proposition}[Green's formula for $\mathcal{H}^{\bullet}$]\label{green2}
\begin{itemize}
	\item [(a)]	Let $\omega\in F^\bullet,\eta\in \F^c.$ Then,
	\begin{align*}
		 		\sum_{\Sigma}m (\mathcal{H}^\bullet\omega)\eta = \sum_\Sigma m\omega(\mathcal{H^\bullet\eta}). 
		 	\end{align*}
	\item[(b)] Let $\omega,\eta\in \F^c.$ Then, 
	\begin{equation*}
			\sum_{\Sigma}m (\mathcal{H}^\bullet\omega)\eta =Q^{\bullet}(\omega,\eta).
	\end{equation*}
	\end{itemize}\end{proposition}
\begin{proof} The statement \emph{(a)} follows directly from  \cite[Lemma~2.1.]{GKS} and for \emph{(b)} we combine this with  Lemma~\ref{schröd} above.
\end{proof}

After we took care of how the action of (the quadratic forms associated to)  $\mathcal{H}^{\bullet}$ relates to the one of our Laplacians, we next come to the issue of their domains.

\begin{lemma}
	We have the following inclusions:
	\begin{itemize}
		\item [(a)] $F^{-}=F^{1}=D(\delta\partial)$.
		\item [(b)] $F^{+}\subseteq D(\partial \delta)$.		
		\item [(c)] $F^{-}\cap F^{+}\subseteq D(\partial\delta+\delta\partial)$.
	\end{itemize}
	If the simplicial complex is locally finite, then also $$\F^{+}=\F^{H}= D(\partial \delta)= D(\partial\delta+\delta\partial)=\F.$$
\end{lemma} 
\begin{proof}
	$(a)$ By the definition of $b^{-}$ we obtain
	\begin{align*}
		F^{-}&=\lbrace \eta\in F\mid\sum_{\tau',\tau\cap\tau'\prec\tau}\frac{m(\tau)m(\tau')}{m(\tau\cap\tau')}\vert\omega(\tau')\vert<\infty~\text{for all}~\tau\in\Sigma\rbrace.
	\end{align*}
	We observe for $\omega\in F$
	\begin{align*}
		\sum_{\tau',\tau\cap\tau'\prec\tau}\frac{m(\tau)m(\tau')}{m(\tau\cap\tau')}\vert\omega(\tau')\vert&=	m(\tau)\sum_{\rho\prec \tau} \frac{1}{m(\rho)} \sum_{\tau'\succ\rho,\tau'\neq\tau} m(\tau')\vert\omega(\tau')\vert.
	\end{align*}
	There are at most finitely many $\rho\prec\tau$ for any $\tau\in\Sigma$, so we have $\inf_{\rho\prec \tau}m(\rho)>0$ and, therefore, the right hand side is finite if and only if $\sum_{\tau'\succ\rho} m(\tau')\vert\omega(\tau')\vert<\infty$ for all $\rho\in \Sigma$. However, by definition this means $\omega\in F^{1}$. Furthermore, since $\partial$ maps $F^{1}=D(\partial)$ to $F=D(\delta)$, we infer $F^{1}=D(\partial)=D(\delta\partial)$. 
	
	$(b) $ By the triangle inequality, we have for $\omega\in F^{
		+}$
	\begin{equation*}
		\sum_{\sigma\succ\tau}  m(\sigma)|\delta \omega(\sigma)|\leq 	\sum_{\sigma\succ\tau}  m(\sigma)\sum_{\tau'\prec \sigma}|\omega(\tau')|=\sum_{\tau'\in \Sigma} b^{+}(\tau,\tau')|\omega(\tau')|+\vert\omega(\tau)\vert\sum_{\sigma\succ\tau}m(\sigma). 
	\end{equation*}
	The right hand side is finite as $\omega\in F^+$ and due to local summability. Hence, $\delta\omega\in F^{1}$ and, therefore, $\omega \in D(\partial\delta)$.
	
$(c)$ This follows now directly from $(a)$ and $(b).$\\
In the locally finite case, we have $ \F^{+}=\F^{H}=\F,$ which yields the result.
\end{proof}

With these preparations, we come to the proof relating Schrödinger operators to Laplacians.

\begin{proof}[Proof of Theorem~\ref{schröd2}] Let $\omega\in \F^{+}$. By the lemma above we have $\omega\in D(\partial\delta)$. We write $\omega=\sum_{\tau}\omega(\tau)1_{\tau}$ and use Green's formulas, Proposition~\ref{green2}~\emph{(a)},~\emph{(b)}, and Theorem~\ref{green1} (c),
	\begin{multline*}
		m(\sigma)\mathcal{H}^{+}\omega (\sigma)= \sum_{\Sigma}m(\mathcal{H}^+\omega )1_\sigma = \sum_\Sigma m\omega(\mathcal{H}^+ 1_\sigma) =\sum_{\tau\in\Sigma}\omega(\tau)Q^{+}(1_{\tau},1_{\sigma})\\=\sum_{\tau\in\Sigma}\omega(\tau)\sum_{\Sigma}m(\partial\delta 1_{\tau})1_{\sigma}=m(\sigma)\partial\delta \omega(\sigma).
	\end{multline*}
The statements for $\mathcal{H}^{-}$ and $\mathcal{H}^{H}$ follow analogously. In the locally finite case, the lemma above gives $\F^{\circ}=\F$.
The formulas follow now directly from Green's formulas, Theorem~\ref{green1}, (a) and (b).
\end{proof}

The theorem, together with Theorem~\ref{action}, readily gives for any Laplacian $$\Delta^{-}=\mathcal{H}^- \qquad \mbox{on }D(\Delta^{-}).$$

Next, we give an example which shows that the inclusion $F^{+}\subseteq D(\partial\delta)$ is strict in general and, therefore, shows that the situation for $\Delta^{+}$ and $\Delta^{H}$
is more complicated.

\begin{example}[$\F^{+}\neq D(\partial\delta)$]
	Let $V=\mathbb{N}\cup\{0,\infty\}$. We consider the finite dimensional simplicial complex whose maximal faces are the triangles $\tau_{n}=\{0,n,\infty\}$, $n\in \mathbb{N}$ and a locally summable weight $m$.
	Let
	\begin{equation*}
		\omega=\sum_{n\in \mathbb{N}}\left(\theta(\{0,n\},\tau_{n}) 1_{\{0,n\}} - \theta (\{n,\infty\},\tau_{n}) 1_{\{n,\infty\}}\right)\frac{1}{m(\tau_{n})}.
	\end{equation*}
	Then, $\delta\omega=0$ and, hence, $\omega\in D(\partial\delta)$. However, 
	\begin{equation*}
		\sum_{n\in \mathbb{N}}b^{+}(\{0,\infty\}, \{n,\infty\})|\omega(\{n,\infty\})|=\sum_{n\in \mathbb{N}}m(\tau_{n})\frac{1}{m(\tau_{n})}=\infty,
	\end{equation*}
	which gives $\omega\not\in \F^{+}$. Indeed, a similar construction of $\omega$ can be carried out in any simplicial complex, where one has  non-locally finiteness on the level of triangles.
\end{example}

In order to study when $\Delta^{+}$ and $\Delta^{H}$ are restrictions of $\mathcal{H}^{+}$ and $\mathcal{H}^{H}$ we employ the theory of magnetic Schrödinger operators developed in the literature. 

We first study when $\F^{c}\subseteq D(\Delta^{\circ})$. We call $m$ \emph{balanced}  if 
\begin{equation*}
	\sup_{\sigma,\tau\in \Sigma,\tau\prec \sigma} \frac{m(\sigma)}{m(\tau)}<\infty.
\end{equation*}

\begin{theorem}\label{Fc} Let $\Delta^{\bullet}$ be a Laplacian for $\bullet\in\{\pm,H\}$. 
	\begin{itemize}
		\item [(a)]  $\F^{c}\subseteq D(\Delta^{-})$.
		\item [(b)] If $\Sigma$ is locally finite or $m$ is balanced, then $\F^{c}\subseteq D(\Delta^{+})\cap D(\Delta^{H})$.
	\end{itemize}
	In this case
	\begin{equation*}
		\Delta^{\bullet}=\mathcal{H}^{\bullet}\qquad \mbox{ on }F^{c}.
	\end{equation*}
\end{theorem}

We first give an abstract characterization which stems from  operators on graphs.

\eat{by the virtue of the functions $\varphi^{\bullet}_{\tau}:  \Sigma\rightarrow [0,\infty)$, $\tau'\mapsto{b^\bullet(\tau,\tau')}/{m(\tau')}$.which are given for $\tau'\neq\tau$
\begin{align*}
	\varphi^{+}_{\tau}(\tau')=\frac{m(\tau\cup\tau')}{m(\tau')}1_{\tau,\tau'\prec\tau\cup\tau'},\qquad
	\varphi^{-}_{\tau}(\tau') = \frac{m(\tau)}{m(\tau\cap\tau')}1_{\tau\cap\tau'\prec\tau,\tau'},\end{align*}
\begin{equation*}
	\varphi^{H}_{\tau}(\tau') = \left\vert	\varphi^{+}_{\tau}(\tau')-	\varphi^{-}_{\tau}(\tau')\right\vert.
\end{equation*}}

\begin{proposition}[$\F^{c}\subseteq D(\Delta^{\circ})$]\label{prop:Fc} 	
	Let $\bullet\in\{\pm,H\}$ and $\Delta^\bullet$ be a Laplacian. If $\bullet = H,$ assume that $Q^H\subseteq Q^{+}_N+Q^{-}_N$.
	The following statements are equivalent:
	\begin{itemize}
		\item [(i)] $\F^{c}\subseteq D(\Delta^{\circ})$ for some (all) Laplacians $\Delta^{\bullet}$.
		\item [(ii)]  $\mathcal{H}^\bullet\F^c\subseteq\ell^2$.
		\item [(iii)] $\varphi^{\bullet}_{\tau}\in \ell^{2}$ for all $\tau\in \Sigma$, where $\varphi^\bullet_\tau(\tau')={b^\bullet(\tau,\tau')}/{m(\tau')}$, $\tau'\in \Sigma$.
		\item [(iv)]  $\ell^{2}\subseteq \F^{\bullet}$.
	\end{itemize}
	This is in particular satisfied if $\Sigma$ is locally finite.
\end{proposition}
\begin{proof}
The equivalence $(ii)\Longleftrightarrow(iii)$ is shown in \cite[Lemma~2.11]{GKS} and the equivalence $(iii)\Longleftrightarrow(iv)$ is shown in \cite[Theorem~1.29]{KLW} as well as the ``in particular'' statement. 

The implication $(i)\Longrightarrow(ii)$ is a consequence of Theorem \ref{action} and Theorem \ref{schröd2}. For $\bullet = \pm$, we get the implication $(iv)\Longrightarrow(i)$ from Theorem~\ref{action} and Green's formula, Theorem \ref{green1}. 

If $\bullet = H$ and we assume $(ii),$ then we get from Theorem~\ref{schröd2} since $F^{c}\subseteq F^{+}\cap F^{-}$ $$(\delta\partial+\partial\delta)F^c=\mathcal{H}^{H}F^{c}\subseteq\ell^2.$$ However, according to Lemma \ref{sammellemma1} and Lemma \ref{sammellemma2} we always have $$\mathcal{H}^-F^c = \delta\partial F^c\subseteq\ell^2.$$ Therefore, we conclude 
\begin{equation*}
	\mathcal{H}^+F^c = \partial \delta F^c\subseteq\ell^2.
\end{equation*}
Thus, $(ii) \Longrightarrow (iv)$ applied to $\bullet=\pm$ implies
that $\ell^2\subseteq F^+\cap F^-.$
Using the assumption $Q^{H}\subseteq Q^{+}_N+Q^{-}_N$, we deduce from  Theorem~\ref{schröd2}, for $\eta\in F^{c}$ and $\omega\in D(Q^{H})\subseteq \ell^{2}\subseteq F^{+}\cap F^{-}$,
\begin{equation*}
	Q^{H}(\omega,\eta)=	Q^{+}_N(\omega,\eta)+	Q^{-}_N(\omega,\eta)=\sum_{\Sigma}\delta\omega \delta\eta+  \sum_{\Sigma}\partial\omega\partial\eta=\sum_{\Sigma} \omega \mathcal{H}^{H}\eta.
\end{equation*}
 $ (ii)$ for $\bullet=H$, gives $\mathcal{H}^{H} \eta\in\ell^{2}$, so  $\eta\in D(\Delta^{H})$ and, therefore, $F^{c}\subseteq D(\Delta^{H})$.
\end{proof}

\begin{proof}[Proof of Theorem~\ref{Fc}]
\emph{(a)} Observe that by the definition of $b^{-}$
\begin{equation*}
	\|\varphi_{\tau}^{-}\|^{2}\leq m(\tau)^{2}\sum_{\rho\prec\tau}\frac{1}{m(\rho)^{2}}\sum_{\tau'\succ \rho}m(\tau')<\infty
\end{equation*}
which follows since there are only finitely many $\rho\prec \tau$ and $\sum_{\tau'\succ \rho}m(\tau')<\infty$ by local summability of $m$. Hence, the statement follows from the proposition above.

\emph{(b)} By the definition of $b^{H}$ we have $\varphi_{\tau}^{H}=|\varphi_{\tau}^{+}-\varphi_{\tau}^{-}|$. Thus, $\varphi_{\tau}^{H}\in \ell^{2}$ if and only if $\varphi_{\tau}^{+}\in\ell^{2}$ since $\varphi_{\tau}^{-}\in\ell^{2}$ by \emph{(a)}. If $\Sigma$ is locally finite, then the statement now follows from the proposition above. Assuming $m$ is balanced  and by the definition of $b^{+}$, we obtain by local summability of $m$
\begin{equation*}
	\|\varphi_{\tau}^{+}\|^{2}\leq\sum_{\sigma\succ\tau}m(\sigma)^{2}
	\sum_{\tau'\prec \sigma}\frac{1}{m(\tau')}\leq (\dim(\tau)+2)\sup_{\tau'\prec \sigma'} \frac{m(\sigma')}{m(\tau')} \sum_{\sigma\succ\tau}m(\sigma) <\infty.
\end{equation*}
Thus, the statement follows again from the proposition above.
\end{proof}

Finally, to obtain equality on the full domain we restrict ourselves to $\Delta^\bullet_{D}$ and employ once more the theory of magnetic Schrödinger operators. 

\begin{theorem}\label{action3}
Assume one of the following two conditions holds:
\begin{itemize}
    \item [(a)] There is $C>0$ such that for all $\varphi\in \F^c$ $$\langle \varphi,-((c^\bullet\slash m)\wedge 0)\varphi\rangle\leq C(\mathcal{Q}^\bullet(\varphi)+\Vert\varphi\Vert^2).$$
    \item[(b)]  $\F^{c}\subseteq D(\Delta^{\circ}_{D})$,  (see proposition above for a characterization).
\end{itemize}
Then $\Delta^\bullet_{D}=\mathcal{H}^\bullet$ on $D(\Delta_{D}^{\bullet}).$
\end{theorem}
\begin{proof}
    This can easily be seen from \cite[Theorem 2.12]{GKS}.
\end{proof}
Note that (a) in the Theorem above is satisfied if $c^\bullet\slash m$ is bounded from below.
\section{Form Uniqueness and Essential self-adjointness}\label{FormeqEssSa}

From the previous section it can be seen that, in general, one cannot speak of \emph{the} Laplacian on an infinite simplicial complex. There is rather a family of Laplacians $\Delta^{\bullet}$ whose quadratic forms $Q^{ \bullet}$ satisfy $Q_{D}^{\bullet}\subseteq Q^{\bullet}\subseteq Q_{N}^{\bullet}$ for each $\bullet\in\{\pm,H\}$. The goal of this section is to study criteria for uniqueness, i.e., when one can speak of \emph{the} Laplacian.

From an operator theoretic point of view, we answer this question in two ways. The first is to give criteria which yield $\Delta_{D}^{\bullet}=\Delta_{N}^{\bullet}$, so there is a unique Laplacian in the sense of Definition~\ref{def:Laplacian}. 
The second is to show essential self-adjointness of the restrictions $\Delta^\bullet_{c}$ of $\Delta^\bullet$ to $\F^c$. However, this requires $\Delta^\bullet_{c}$ to be a symmetric operator on $\ell^{2},$ for which we gave criteria in Theorem~\ref{Fc} and Proposition~\ref{prop:Fc}. This also excludes the existence of more self-adjoint operators which are restrictions of $\delta\partial$, $\partial\delta$ and $\delta\partial+\partial\delta$ but may not have quadratic forms that are restrictions of $Q^{\bullet}_{N}$. To study these questions, we heavily draw from the theory of magnetic Schrödinger operators, to which we bridged the gap in Section~\ref{DiscSch}.

We start by characterizing boundedness in the next section, which trivially gives the desired uniqueness. In the section that follows, we provide various abstract criteria which are the foundation of the concrete results. Such are criteria on the simplicial complex as a measure space involving the curvature and as metric space.

\subsection{Boundedness}\label{boundedness}
In this section, we characterize boundedness of the operators. For graphs, it is well known that boundedness of the graph Laplacian is equivalent to boundedness of the vertex degree. Here, we identify quantities which play a similar role. For $\bullet\in\{\pm,H\}$, we let $\gamma^{\bullet}(\tau)=\Q^{\bullet}(1_{\tau})$ which is given for $\bullet=+$ by
\begin{align*}
	\gamma^+(\tau) = \sum_{\sigma\succ \tau}m(\sigma).
\end{align*}

We obtain a characterization for finite dimensional complexes and a sufficient criterion in general for boundedness of the operators.
\begin{theorem}[Boundedness]\label{bounddim} 
	Let $\Delta^{\bullet}$ be a Laplacian.
	\begin{itemize}
		\item [(a)]	If $\sup_\Sigma(\dim+1)\gamma^+\slash m<\infty$, then $\Delta^\bullet$ is bounded for all $\bullet\in\{\pm,H\}$. 
		\item [(b)]If  $\Sigma$ is finite dimensional, then the following are equivalent:
		\begin{itemize}
			\item [(i)]  $\Delta^{\bullet}$ is bounded for some (all) $\bullet\in\{\pm,H\}$.
			\item[(ii)] $\gamma^{+}\slash m$ is bounded.
		\end{itemize} 
	\end{itemize}
	
\end{theorem}
It is remarkable that one can characterize the  boundedness of several Laplacians $\Delta^{\bullet}$ arising from forms $Q^{\bullet}_{D}\subseteq Q^{\bullet}\subseteq Q^{\bullet}_{N}$ simultaneously. 

Theorem~\ref{bounddim}  will be deduced from the following proposition, which provides a characterization even if $\Sigma$ is infinite dimensional. On the downside, this characterization includes an assumption on the lower semiboundedness of the quadratic form $q^{\bullet}$ on $\F^c$ given by
$$q^\bullet(\omega) = 2\sum \gamma^\bullet\omega^2-\Q^\bullet(\omega),$$
which is far less explicit than boundedness of the function $\gamma^{+}/m$. The next proposition is essentially a direct application of \cite[Theorem~4.1]{S} for magnetic Schrödinger operators.

\begin{proposition}\label{bound} Let $Q^{\bullet}_{D}\subseteq Q^{\bullet}\subseteq Q^{\bullet}_{N}$ be closed forms given with associated Laplacian $\Delta^{\bullet}$, $\bullet\in\{\pm,H\}$. 
    Then, the following statements are equivalent:
    \begin{itemize}
        \item[(i)] $Q^\bullet$ (respectively $\Delta^\bullet$) is bounded on $\ell^2$ for some (all) $\bullet\in\{\pm,H\}$.
        \item[(ii)]  $\gamma^\bullet\slash m$ is bounded and $q^\bullet$ is lower semibounded for some (all) $\bullet\in\lbrace \pm,H\rbrace$.
    \end{itemize}
\end{proposition}
\begin{proof}
	The equivalence of the form and the operator follow immediately from \cite[Theorem 5.35.]{W} and the Cauchy-Schwarz inequality together with density of the operator domain in the form domain.

    We start by showing that all $Q^{\bullet}$ are simultaneously bounded. By definition of $Q^+$ this precisely means that $\delta_{D}=\delta_{N}=\delta\vert_{\ell^{2}}$. Hence, in view of Lemma~\ref{d*}, $\delta_{N}^{*}=\partial_{D}$ is bounded and $\partial_{D}=\partial_{N}=\partial\vert_{\ell^{2}}$.
    This precisely means that $Q^-$ is bounded. On the other hand, if $Q^-$ is bounded, then the boundedness of $Q^+$ follows along  the same line of reasoning.\\
    If $Q^H$ is bounded, then $Q^H = Q^++Q^-.$ As all forms are positive, this implies the boundedness of $Q^+$ and $Q^-.$ Likewise, the boundedness of $Q^+$ and $Q^-$ implies that $Q^H$ is bounded.\\
    Furthermore,  the argument above shows that all forms $Q^{\bullet}_{D}\subseteq Q^{\bullet}\subseteq Q^{\bullet}_{N}$ are bounded simultaneously, so we can assume $Q^{\bullet}=Q^{\bullet}_{D}$.
    
     The equivalence is now a direct application of \cite[Theorem~4.1]{S} by invoking Lemma~\ref{schröd}: For the convenience of the reader, we discuss how to translate their notation to the one of this paper. By Lemma~\ref{schröd} our form $\Q^{\bullet}_{c}=Q^{\bullet}\vert_{F^{c}}$ is a form of type $\Q^{c}_{W,\Phi,E}$ from \cite{S}     with $b=b^{\bullet}$, $W=c^{\bullet}/m$, $\mu=m$, $E=(E_{\tau})_{\tau\in\Sigma}$ being chosen as $E_{\tau}=\mathbb{R}$ and $\Phi_{\tau,\tau'}z=o^{\bullet}(\tau,\tau')z$ for $\tau,\tau'\in \Sigma$ and $z\in \mathbb{R}$.
    A direct calculation shows that
    \begin{equation*}
    	\gamma^{\bullet}(\tau)= \sum_{\tau'\in \Sigma}{b^\bullet(\tau,\tau')}+{c^\bullet(\tau)} 
    \end{equation*}
    so $\gamma^{\bullet}$ coincides with $\deg$ and $\gamma^{\bullet}/m$ with $B$ from \cite[Theorem~4.1]{S}.  Now, \cite[Theorem~4.1~(i)]{S} gives that $\Q^{\bullet}_{c}=\Q^{c}_{W,\Phi,E}$ is bounded if and only if $\gamma^{\bullet}/m=B$ is bounded and  $c^{\bullet}/m=W\in \mathcal{A}_{\mu,\Phi,E}\cap\mathcal{A}_{\mu,-\Phi,E} $ where $W\in \mathcal{A}_{\mu,\Phi,E}$ means that the corresponding form $\mathcal{Q}^{c}_{W,\Phi,E}$ is bounded from below. Clearly, our forms $\Q_{c}^{\bullet}$ are bounded from below. To see that $W\in \mathcal{A}_{\mu,-\Phi,E}$ is equivalent to $q^{\bullet}$ being bounded from below, we invoke Equation~($\heartsuit$) from \cite{S} which finishes the proof.
\end{proof}

With this at hand, we come to the proof of the boundedness result.

\begin{proof}[Proof of Theorem~\ref{bounddim}] \emph{(a)}
    If $(\dim+1)\gamma^{+}\slash m$ is bounded, then so is $\gamma^{+}\slash m.$ By Proposition~\ref{bound} we only need to show semiboundedness of $q^+.$ This follows from Lemma~\ref{schröd}, positivity of $Q^+_{D},$ the Cauchy-Schwarz inequality and Fubini's theorem
    \begin{multline*}
        q^+(\omega)  = Q^+(\omega) +2\sum_{\tau,\tau'\Sigma}b^+(\tau,\tau')o^+(\tau,\tau')\omega(\tau)\omega(\tau')
        \geq-2\sum_{\tau,\tau'\in \Sigma}b^+(\tau,\tau')\omega^2(\tau)\\= -2\sum_{ \tau\in\Sigma}\omega^2(\tau)\sum_{\sigma\succ\tau}m(\sigma)(\dim(\tau)+1)=-2\sum_\Sigma (\dim(\tau)+1)\gamma^{+}\omega^2.
    \end{multline*}
  Now \emph{(b)} follows from the same arguments given above.
\end{proof}

\subsection{Some Abstract Criteria}

In this section we collect various abstract criteria to show form equality $Q^{\bullet}_{D}=Q^{\bullet}_{N}$ which is equivalent to $\Delta^{\bullet}_{D}=\Delta^{\bullet}_{N}$. 

The first equivalence relates form equality of $Q^{H}$ to the essential self-adjointness of the Gauss-Bonnet operator $\delta+\partial$ on $F^{c}$. This shows that the results later in this paper also generalize the ones in \cite{ATH}.

\begin{proposition}[Gauss-Bonnet Operator]\label{d+del}
	The following are equivalent:
	\begin{itemize}
		\item[(i)] $Q^H_{D} = Q^H_N$
		\item[(ii)] $(\delta+\partial)_D = (\delta+\partial)_N$
		\item[(ii)] The Gauss-Bonnet operator $(\delta+\partial)\mid_{\F^c}$ is essentially self-adjoint.
	\end{itemize}
	In that case, we have $$(\delta+\partial)_D = \delta_D+\partial_D = \delta_D+\partial_N =\delta_N+\partial_D =\delta_N+\partial_N =(\delta+\partial)_N.$$
	
\end{proposition}
\begin{proof}
	$(i) \Leftrightarrow (ii)$: This is clear.\\
	$(ii) \Rightarrow (iii)$:
	If $(\delta+\partial)_D = (\delta+\partial)_N$ then $(\delta+\partial)_D$ is self-adjoint by Lemma \ref{d*}. Therefore, $(\delta+\partial)\mid_{\F^c}$ is essentially self-adjoint as it is the closure of $(\delta+\partial)_D$.

	$(iii) \Rightarrow (ii)$: Essential self-adjointness of $(\delta+\partial)\vert_{\F^c}$ implies that its closure $(\delta+\partial)_D$ is self-adjoint. Thus, by Lemma \ref{d*} $(\delta+\partial)_D = (\delta+\partial)_N.$ 
	
	The chain of equalities follows because we always have the operator inclusions
	$$(\delta+\partial)_D \subseteq \delta_D+\partial_D \subseteq \delta_D+\partial_N \subseteq \delta_N+\partial_N \subseteq(\delta+\partial)_N$$ as well as 
	\begin{equation*}
		(\delta+\partial)_D\subseteq \delta_N+\partial_D\subseteq(\delta+\partial)_N.\hfill\qedhere
	\end{equation*}
\end{proof}

We now relate form equality for $+$, $-$ and $H$.
\begin{proposition}\label{Q+Q-}The following statements are equivalent:
	\begin{itemize}	\item[(i)] $Q^+_{D} = Q^+_N$.
			\item[(ii)] $Q^-_{D} = Q^-_N$.
	\end{itemize}
	In this case, $Q_{\Sigma} = Q_\Sigma'$ from Section~\ref{Hilbert}.
\end{proposition}
\begin{proof}
	$(i)$  is equivalent to $\delta_{D}=\delta_{N}$ which is, by Lemma~\ref{d*}, equivalent to $\partial_{N}=(\delta_{D})^{*}=(\delta_{N})^{*} = \partial_D$ which is equivalent to $(ii)$.
	
	The quadratic forms for $Q_{\Sigma}$ and $Q_\Sigma'$ from Section~\ref{Hilbert} are given by $\delta_{N}+\partial_{D}$ and $\delta_{D}+\partial_{N}$. Hence, the statement follows directly from the first part of the proof.
\end{proof}

\begin{remark}The proposition raises the question whether one can deduce anything for $Q^{H}$. This however, seems not possible by examining the first chain of inclusions in the proof of Proposition~\ref{d+del}. Here $\delta_{D}=\delta_{N}$ and $\partial_{D}=\partial_{N}$ gives equality of all operators in the middle, however, not on the outer two fringes of the inclusion.
\end{remark}

Next, we come to a criterion which relates form uniqueness to validity of a certain Green's formula and existence of $\alpha$-harmonic forms. This will be heavily used in the next subsection. To streamline notation, we denote formal Laplacians as follows
\begin{align*}
	\mathcal{L}^+ &= \partial\delta\quad&&\mbox{on }D(\mathcal{L}^+) = D(\partial\delta)=\lbrace\omega\in \F\mid\delta\omega\in \F^1\rbrace,\\
	\mathcal{L}^- &= \delta\partial\quad&&\mbox{on }D(\mathcal{L}^-) = D(\delta\partial)=\F^1,\\
	\mathcal{L}^H &= (\partial+\delta)^2\quad&&\mbox{on }D(\mathcal{L}^H) =  D((\delta+\partial)^{2})=\lbrace \omega\in \F^1\mid (\partial+\delta)\omega\in \F^1\rbrace.\\
\end{align*}
Since $\ell^{2}\subseteq F^{1}$ by Lemma~\ref{sammellemma2}, it is clear that $$D(Q^{\bullet}_{N})\subseteq D(\mathcal{L}^{\bullet}).$$ 
The criterion below is a variation of \cite[Theorem 3.2]{KLW}. However, we cannot  apply the results of Section~\ref{DiscSch} directly, since the setting of \cite{KLW} does not include magnetic potentials nor negative electric potentials.

\begin{proposition}[Form uniqueness]\label{formeq}
    The following are equivalent:
    \begin{itemize}
        \item[(i)] $Q^\bullet_{D} = Q^\bullet_N.$
        \item[(ii)] $D(\Delta^\bullet_{D}) = \lbrace \omega\in D(Q^\bullet_N)\mid\mathcal{L}^\bullet\omega\in\ell^2\rbrace.$
        \item[(iii)] For all $\omega,\eta\in D(Q^\bullet_N)$ and $\mathcal{L}^\bullet\omega\in\ell^2$ one has $Q^\bullet_N(\omega,\eta) = \langle\mathcal{L}^\bullet\omega,\eta\rangle.$
        \item[(iv)] For every $\omega\in D(Q^\bullet_N),\alpha>0$ with $(\mathcal{L}^\bullet+\alpha)\omega = 0$ one has $\omega\equiv 0.$
       \end{itemize}
\end{proposition}

\begin{proof}
	$(i)\Rightarrow (ii)$: Clearly, $D(\Delta^\bullet_D)\subseteq\lbrace \omega\in D(Q^\bullet_N)\mid\mathcal{L}^\bullet\omega\in\ell^2\rbrace=:W^{\bullet}$ since $\Delta^\bullet_D$ is a restriction of $\mathcal{L}^\bullet$ according to Theorem \ref{action}. 
	We show $\omega\in D(\Delta^{+}_D)$ for $\omega \in W^{+}.$ The other cases $\bullet=-,H$ follow analogously.  By $(i)$ we have $\delta_{D}=\delta_{N}$, so by Corollary~\ref{LNLD}, we have $\partial_N\delta_D = \Delta^+_D= \partial_N\delta_N.$
	So, $\omega \in W^{+}$ precisely means that $\delta\omega\in\ell^2$ and $\partial\delta\omega\in\ell^2.$
	Thus, $\omega\in D(\partial_N\delta_N) = D(\Delta^+).$\\
	$(ii)\Rightarrow (iii)$: Since $D(\Delta^\bullet_N)\subseteq W^{\bullet} = D(\Delta^\bullet_D)$, we see that $\Delta^\bullet_N\subseteq\Delta^\bullet_D$ and as both are self-adjoint operators this already means $\Delta^\bullet_D = \Delta^\bullet_N.$ According to Theorem \ref{action} both are restrictions of $\mathcal{L}^\bullet$ and so for $\omega,\eta\in D(Q^\bullet_N)$ with $\mathcal{L}^\bullet\omega\in\ell^2$, we conclude from $(ii)$ that $\omega\in D(\Delta^\bullet_N)$ and $$Q^\bullet_N(\omega,\eta) = \langle\Delta^\bullet_N\omega,\eta\rangle = \langle\mathcal{L}^\bullet\omega,\eta\rangle,$$ thus proving $(iii)$.\\
    	$(iii)\Rightarrow (iv)$: Let $\alpha>0$, $\omega\in D(Q^\bullet_N)$ and $(\mathcal{L}^\bullet+\alpha)\omega = 0.$ Clearly, $\mathcal{L}^\bullet\omega = -\alpha\omega\in\ell^2$ and so $(iii)$ implies
    $$0\leq Q^\bullet_N(\omega) = \langle\mathcal{L}^\bullet\omega,\omega\rangle = -\alpha\Vert\omega\Vert^2\leq 0.$$ This can only be true if $\omega\equiv 0,$ i.e., $(iv)$ holds.\\
    $(iv)\Rightarrow (i)$: For $\omega\in D(\Delta^\bullet_N)$, $\alpha>0$  let  $\eta = (\Delta^\bullet_{D}+\alpha)^{-1}(\Delta_{N}^\bullet+\alpha)\omega\in D(\Delta^{\bullet}_{D})\subseteq D(Q^{\bullet}_{N}).$ Since $\Delta^{\bullet}_{D}$ and $\Delta^{\bullet}_{N}$  are restrictions of $\mathcal{L}^{\bullet}$,  $$(\mathcal{L}^\bullet+\alpha)(\eta-\omega) = 0.$$
    If $D(Q^\bullet_{D})\neq D(Q^\bullet_N)$ in which case $\Delta^\bullet_{D}\neq\Delta^\bullet_N$, then there is $\omega\neq \eta$.
\end{proof}

Next, we discuss  how form uniqueness of $Q^{H}$ implies form uniqueness of $Q^{\pm}$.

\begin{corollary}\label{QH->Q+Q-}
	If  $ Q^H_{D} = Q^H_N,$ then
	\begin{equation*}
		Q^+_{D} = Q^+_N\qquad\mbox{and}\qquad Q^-_{D} = Q^-_N.
	\end{equation*}		
\end{corollary}
\begin{proof}Let $\omega\in D(Q^-_N),\alpha>0$ be such that $\mathcal{L}^-\omega =-\alpha\omega.$ Then applying $\delta$ yields $$\delta\omega = -\alpha^{-1}\delta\delta\partial\omega = 0,$$
which implies $\omega\in D(Q^{+}_{N})$.	Thus, if $Q^H_{D} = Q^H_N$, then Proposition~\ref{d+del} implies 
$$\omega\in D(Q^+_N)\cap D(Q^-_N) = D(\delta_N+\partial_N) = D((\delta+\partial)_N) =D(Q^H_N).$$  As a consequence  and since  $ Q^H_{D} = Q^H_N$ implies  $\partial\partial =0$ on $D(Q^{H}_{N})=D(Q^{H}_{D})\subseteq D(Q^{-}_{D})$ by Lemma~\ref{pp01}, we get
	$$\mathcal{L}^H\omega = (\delta+\partial)^2\omega = \delta\partial\omega =\mathcal{L}^{-}\omega= -\alpha\omega.$$
	By Theorem \ref{formeq} we have $\omega = 0$ and so $Q^-_{D} = Q^-_N$ again by Theorem \ref{formeq}. Proposition~\ref{Q+Q-} implies that also $Q^+_{D} = Q^+_N.$	
\end{proof}

After discussing form uniqueness, we now address essential self-adjointness of $\Delta^\bullet_{c}=\mathcal{L}^{\circ}\vert_{F^{c}}$. Recall that  $F^c\subseteq D(\Delta^\bullet_{D})$ implies (and is indeed equivalent to) $$\ell^{2}\subseteq D(\mathcal{H}^{\bullet})$$
by Green's formula, Theorem~\ref{green1}, and the definition of $D(\Delta^\bullet_{D})$, cf. the proof of
 Proposition~\ref{prop:Fc}.  Furthermore, recall that Theorem~\ref{Fc} provides criteria for $F^c\subseteq D(\Delta^\bullet_{D})$ which always hold for $\bullet=-$ and  for $\bullet=+,H$ if $\Sigma $ is locally finite or balanced. The next proposition is a variant of \cite[Theorem~3.6]{KLW}.
 
\begin{proposition}[Essential self-adjointness]\label{essSA}
Assume $F^c\subseteq D(\Delta^\bullet_{D}).$ Then the following are equivalent:
    \begin{itemize}
        \item[(i)] $\Delta^\bullet_{c}$ is essentially self-adjoint.
        \item[(ii)] $D(\Delta^\bullet_{D}) = \lbrace \omega\in\ell^2\mid\mathcal{H}^\bullet\omega\in\ell^2\rbrace.$
        \item[(iii)] For some (all) $\alpha>0$ and $\omega\in\ell^2$ such that $\mathcal{H}^\bullet \omega= -\alpha \omega$ we have $\omega\equiv 0.$
    \end{itemize}
\end{proposition}
\begin{proof}
In the proof we will subsequently use the fact that $\Delta^\bullet_c = \mathcal{H^\bullet}$ on $F^c$ and so by Green's formula, Proposition~\ref{green2}~\emph{(a)},
  $$D\big((\Delta^\bullet_{c})^*\big) =  \lbrace \omega\in\ell^2\mid\mathcal{H}^\bullet\omega\in\ell^2\rbrace \quad \mbox{and}\quad (\Delta^\bullet_{c})^*=\mathcal{H}^{\bullet}\;\mbox{on }D\big((\Delta^\bullet_{c})^*\big).$$
Further, $\Delta^\bullet_{c}$ is essentially self-adjoint if and only if $(\Delta^\bullet_{c})^*$ is, \cite[Theorem~5.20]{W}. This in turn 
is equivalent to $ \Delta^\bullet_{D}= (\Delta^\bullet_{c})^*$ since 
$  \Delta^\bullet_{D}\subseteq (\Delta^\bullet_{c})^*$ and $ \Delta^\bullet_{D}$ is self-adjoint.

   $(i)\Rightarrow(ii)$: If 
   $  \Delta^\bullet_{D}= (\Delta^\bullet_{c})^*$, then $(ii)$ follows from the discussion above.

    $(ii)\Rightarrow(iii)$: Let $\alpha>0$ and $\omega\in\ell^2$ such that $\mathcal{H}^\bullet \omega = -\alpha \omega$. Then, $\mathcal{H}^\bullet \omega\in\ell^2$ and by $(ii)$ we conclude  $\omega\in D(\Delta^\bullet_{D}).$ But then
    $$0\leq Q^\bullet_{D}(\omega) = \langle\Delta^\bullet_D \omega,\omega\rangle = -\alpha\Vert \omega\Vert^2$$
    which can only be true if $\omega\equiv 0.$
    
      $(iii)\Rightarrow(i)$:  
      Let $\omega\in D((\Delta^\bullet_{c})^*)\subseteq F^{\bullet}$, $\alpha>0$ and
      $$\eta = (\Delta^\bullet_{D}+\alpha)^{-1}((\Delta^\bullet_{c})^*+\lambda)\omega= (\Delta^\bullet_{D}+\alpha)^{-1}(\mathcal{H}^\bullet+\lambda)\omega\in D(\Delta^\bullet_{D}).$$ 
      If $ \Delta^\bullet_{D}\neq (\Delta^\bullet_{c})^*$, then there is  $\omega\neq \eta$. On the other hand,  since $\Delta_{D}^{\bullet}$ is a restriction of $\mathcal{H}^{\bullet}$ by Theorem~\ref{action3},  we have
      $(\mathcal{H}^\bullet+\lambda)(\omega-\eta) = 0$ which contradicts $(iii)$.
\end{proof}

 We finish the section by showing that essential self-adjointness of $\Delta^\pm_c$ can be deduced from the essential self-adjointness of $\Delta^H_c.$

\begin{proposition}\label{esssa2}
   If $F^{c}\subseteq D(\Delta_{D}^{H})$ and $\Delta^H_{c}$ is essentially self-adjoint, then $\Delta^\pm_{c}$ is essentially self-adjoint. 
\end{proposition}
\begin{proof}We have that $F^{c}\subseteq D(\Delta_{D}^{H})$ implies $F^{c}\subseteq D(\Delta_{D}^{+})$ (cf. Proposition~\ref{prop:Fc} or the proof of Theorem~\ref{Fc}) and we always have $F^{c}\subseteq D(\Delta_{D}^{-})$ by Theorem~\ref{Fc}. Thus, $\ell^2\subseteq F^+\cap F^-$ and $\mathcal{H}^H = \delta\partial+\partial\delta $ on $\ell^2$ by Proposition \ref{prop:Fc} and Theorem~\ref{schröd2}. Now, observe for $\omega\in F^{c}$
\begin{equation*}
	\|(\delta\partial+\partial\delta)\omega\|^{2}-	\|\delta\partial\omega\|^{2}-
	\|\partial\delta\omega\|^{2}= 2\langle \delta\partial \omega,\partial\delta \omega\rangle=2\langle \partial \omega,\partial\partial\delta \omega\rangle=0,
\end{equation*}	
where the second equality holds by Stokes' theorem, Theorem~\ref{adj1} as  $\partial\omega\in F^{c}$ and the last equality holds by Lemma~\ref{pp0} as $\delta\omega\in F^{2}$ by Lemma~\ref{sammellemma2}~(d).
Thus, we have
	\begin{equation*}
		\|\Delta_{c}^{\pm}\omega\|^{2} \leq  \|\delta\partial\omega\|^{2}+
		\|\partial\delta\omega\|^{2}=	\|(\delta\partial+\partial\delta)\omega\|^{2}=\|\Delta_{c}^{H}\omega\|^{2},
	\end{equation*}
	which means that the operators $ \Delta_{c}^{\pm}$ both are $\Delta_{c}^{H}$-bounded with bound $1$ and  symmetric by Green's formula, Proposition~\ref{green1}. Thus, the statement follows from \cite[Theorem~4.6]{K} and the essential self-adjointness of $\Delta_{c}^{H}$.
\end{proof}

\subsection{A Weight and Curvature Criterion}
In this section we give a criterion arising from the weight and curvature. Recall to this end that the quantity $c^{H}/m$ introduced in Section~\ref{DiscSch} is referred to as Forman curvature.

\begin{theorem}\label{esssa1}
	If $m$ is uniformly positive and the Forman curvature $c^{H}/m$ is bounded from below,     then $\Delta^\bullet_{c}$ is essentially self-adjoint for all $\bullet\in\{\pm,H\}$.
\end{theorem}

To prove this, we apply the theory of Schrödinger operators on graphs. By Section~\ref{DiscSch} we have the graphs $b^{\bullet}$ associated to $Q^{\bullet}$. We call a sequence of elements of $\Sigma$ a path if it is a path with respect to $b^{\bullet}$.
Recall the quantities $\gamma^{\bullet}(\tau)=Q^{\bullet}(1_{\tau})$, $\tau\in \Sigma$, introduced to characterize boundedness in Section~\ref{boundedness}. The following criterion is based on \cite{Go,GKS}. To this end observe that for $\tau\in\Sigma$
 $$\gamma^\bullet(\tau)-c^\bullet (\tau)=Q^{\bullet}(1_{\tau})-c^{\bullet}(\tau)=\sum_{\tau'\in\Sigma}b^{\bullet}(\tau,\tau')\ge0.$$

\begin{proposition}\label{measureEss}Let $\bullet\in\{\pm,H\}$.
 If there is $\alpha\in\mathbb{R}$ such that for every infinite path $(\tau_n)$ with respect to $b^{\bullet}$  one has
	$$\sum_{n = 1}^{\infty}m(\tau_n)\prod_{j = 0}^{n-1}\left(1+\frac{c^\bullet(\tau_j)-\alpha m(\tau_j)}{\gamma^\bullet(\tau_j)-c^\bullet(\tau_j)} \right)^2 = \infty,$$ then $Q^{\pm}_{D}=Q^{\pm}_{N}$ whenever $\bullet=\pm$. If, in addition, $\F^c\subseteq D(\Delta^\bullet_D)$, then $\Delta^\bullet_{c}$ is essentially self-adjoint and, in particular, $Q^{\bullet}_{D}=Q^{\bullet}_{N}.$
\end{proposition}
\begin{proof}
	The proof of \cite[Theorem 2.14]{GKS} shows that under the present assumption there are no non-trivial solutions $\omega \in F^{\bullet}\cap \ell^{2}$ such that $(\mathcal{H}^{\bullet}-\lambda)\omega=0$ for $\lambda$ smaller than a certain $\lambda_{0}$. Since $\mathcal{H}^{-}=\mathcal{L}^{-}$, we conclude $Q^{-}_{D}=Q^{-}_{N}$ from Proposition~\ref{formeq} and $Q^{+}_{D}=Q^{+}_{N}$ from Proposition~\ref{Q+Q-}. Moreover, we have $\mathcal{H}^{\bullet}=\mathcal{L}^{\bullet}$ on $F^{c}$ if $ \F^c\subseteq D(\Delta^\bullet_D)$ by Proposition~\ref{prop:Fc} and Theorem~\ref{action}. Thus, the essential self-adjointness follows from Proposition~\ref{essSA}.
\end{proof}

\begin{proof}[Proof of Theorem~\ref{esssa1}]
	According to Theorem~\ref{Fc} we have $\F^c\subseteq D(\Delta^\bullet_D).$ We assume $c^\bullet \ge -Km$ for some $K>0$. Choosing $\alpha=-K$, we find that 
	$$\frac{c^\bullet+Km }{(\gamma^\bullet-c^\bullet)}  \ge0$$
	since $\gamma^\bullet-c^\bullet \ge0$. 
	 Thus, the statement for $\bullet=H$ is an easy consequence of Proposition~\ref{measureEss} as uniform positivity of $m$ implies local finiteness and, thus, $F^{c}\subseteq D(\Delta^{H}_{D})$ by Theorem~\ref{Fc}. The statement for $\bullet=\pm$ now follows from Proposition~\ref{esssa2}.
\end{proof}

\begin{remark}
	From the proof of Theorem~\ref{esssa1} it is clear that essential self-adjointness of $\Delta^{\pm}_{c}$ can also be deduced from a lower bound on $c^{\pm}/m$.
\end{remark}

\subsection{A Metric Space Criterion}\label{metricsection}
In this section we show a criterion for form uniqueness and essential self-adjointness in terms of a metric on the 1-skeleton of the complex. This stands in analogy to a theorem of Gaffney \cite{G,Roe,Chernoff,Strichartz} that on a complete Riemannian manifold all Laplacians on forms are essentially self-adjoint. To this end, we define a metric on the 1-skeleton which itself is a graph. Here we have the usual adjacency relation $\sim$ between vertices if they are connected by an edge, i.e., contained in a common $1$-simplex. We will define a path metric on the 1-skeleton of $\Sigma$ as follows.
Let $v,v'\in\Sigma_0.$ We say that $v$ and $v'$ are connected if $b^+(v,v')>0.$ In that case we will write $v\sim v'.$ We then define 
$$d^{\bullet}(v,v') = \inf_{v=v_{0}\sim\ldots\sim v_{n}=v'} \sum_{k=0}^{n-1}w^{\bullet}(v_k,v_{k+1}),$$
where  $w^{\bullet}(x,x) = 0$ and for $x\sim x'$ 
\begin{align*}
	w^{+}(x,x') &= \inf_{\tau\cap \{x,x'\}\neq\varnothing}\Big(\frac{\dim(\tau)+1}{m(\tau)}\sum_{\sigma\succ\tau}m(\sigma)\Big)^{-1/2},\\
w^{-}(x,x') &= \inf_{\tau\cap \{x,x'\}\neq\varnothing}\frac{1}{2}
\Big(\sum_{\tau'\succ\tau'\cap \tau}\frac{m(\tau')}{m(\tau\cap{\tau'})}\Big)^{-1/2},\\
w^H(x,x') &= \frac{1}{\sqrt{2}}\min\lbrace w^+(x,x'),w^-(x,x')\rbrace.
\end{align*}
This yields path pseudo metrics on $\Sigma_0.$ 

\begin{remark}
	The motivation to study this metric stems from the concept of intrinsic metrics, see e.g. \cite[Chapter 11]{KLW} and references therein and has  shown to be vital for a Gaffney  theorem for graphs
	, see \cite{HKMW} or \cite[Chapter~12.3]{KLW}.
\end{remark}

For Riemannian manifolds metric or geodesic completeness is equivalent to compactness of closed metric balls by the Hopf-Rinow theorem. For graphs this remains true in the locally finite case \cite[Theorem~11.16]{KLW}.

\begin{theorem}[Gaffney-type theorem]\label{metric} 

If for $d^{\bullet}$, $\bullet\in \{\pm,H\}$, all metric balls are finite and $\Sigma$ is locally finite or balanced, then $\Delta^\bullet_{c}$ is essentially self-adjoint.
	
\end{theorem}

\begin{remark}
	In \cite{ATH} the authors use a different notion of completeness, that is the existence of certain cut-off functions. However, this is equivalent to the existence of an intrinsic metric with finite metric balls, \cite[Theorem~11.31]{KLW}.
\end{remark}

For the proof, we draw again from the theory of operators on graphs. The following proposition is a variation of \cite[Theorem~2.16]{GKS} or \cite[Theorem~2]{HKMW}. 

\begin{proposition}\label{intrinsic}
 Let $\bullet\in\{\pm,H\}$ and let $d$ be an intrinsic metric, i.e. a pseudo metric such that $\sum_{\tau'}b^\bullet(\tau,\tau')d^2(\tau,\tau')\leq m(\tau)$ for all $\tau\in\Sigma.$ Assume that all metric balls are finite. If $\F^c\subseteq D(\Delta^\bullet_D),$ then $\Delta^\bullet_{c}$ is essentially self-adjoint.
\end{proposition}
\begin{proof}By restricting our attention to connected components, we assume that the graph $b^\bullet$ is connected. By Proposition \ref{prop:Fc} we know that $\ell^2\subseteq F^\bullet.$
	Let $u\in \ell^2$ and $\lambda>0$ such that $\mathcal{H}^\bullet u+\lambda u = 0.$ For $g\in F^c$ we have $$ Q^\bullet_N(ug)+\lambda\Vert ug\Vert^2 = Q^\bullet_D(ug)+\lambda\Vert ug\Vert^2 = Q^{(u)}(g)$$ with
	$$Q^{(u)}(g) = \frac{1}{2}\sum_{\tau,\tau'}b^\bullet(\tau,\tau')o^\bullet(\tau,\tau')u(\tau)u(\tau')\vert g(\tau)-g(\tau')\vert^2.$$ This follows from a direct computation (see also \cite[Proposition 2.18]{GKS}). Now let $R>0$ and fix $\tau_0\in\Sigma.$ Define $$\eta_R(\tau) = 1\wedge\frac{(2R-d(\tau,\tau_0))_+}{R}.$$ By assumption all $\eta_R$ have finite support and $1_R\leq\eta_R\leq 1$ where $1_R$ is the characteristic function on the metric ball with radius $R.$ We get
	\begin{multline*}
		 Q^\bullet(u\eta_R)+\lambda\Vert u\eta_R\Vert^2 = Q^{(u)}(\eta_R)
		\leq\frac{1}{2}\sum_{\tau,\tau'\in \Sigma}b^\bullet(\tau,\tau')(u(\tau)^2+u(\tau')^2)\vert\eta_R(\tau)-\eta_R(\tau')\vert^2\\
		 \leq\frac{1}{R^2}\sum_{\tau\in \Sigma} u^2(\tau)\sum_{\tau'\in \Sigma}b^\bullet(\tau,\tau')d^2(\tau,\tau')
		\leq\frac{1}{R^2}\Vert u\Vert^2.
	\end{multline*}
	Letting $R$ tend to infinity, we see that $\Vert u\eta_R\Vert$ has to vanish in the limit since the left hand side is greater than zero. The statement now follows from Theorem \ref{essSA}.
\end{proof}

We come to the proof of Theorem~\ref{metric}.

\begin{proof}[Proof of  Theorem~\ref{metric}] We want to apply the proposition above. We have to show that the metric $d^\bullet $ gives rise to an intrinsic metric on the entire simplicial complex.
	
Inspired by the Hausdorff distance between subsets of metric spaces, we define the pseudo metrics $d_{h}^{\circ}:\Sigma\rightarrow[0,\infty)$ as
	$$d_h^{\circ}(\tau,\tau') = \max\lbrace \max_{v\in\tau}\min_{v'\in\tau'}d^{\circ}(v,v'),\max_{v'\in\tau'}\min_{v\in\tau}d^{\circ}(v,v')\rbrace.$$
	
If all distance balls on the  $1$-skeleton are finite, then, as consequence of the fact that every simplex has only finitely many vertices, it follows easily that all distance balls on $\Sigma$ with respect to $d_h^\circ$ are finite as well. Moreover, it turns out that $d_h^\circ$ is intrinsic: First of all, if $b^{\circ}(\tau,\tau')>0$  for $\tau,\tau'\in\Sigma$, then
\begin{align*}
	d_h^{\circ}(\tau,\tau') = \max\{\min_{v'\in \tau'}d^{\circ}(\tau\setminus \tau',v'),\min_{v\in \tau}d^{\circ}(v,\tau'\setminus \tau)\}\leq  d^{\circ}(\tau\backslash\tau',\tau'\backslash\tau)
\end{align*}
\eat{ if $b^{+}(\tau,\tau')>0$  for $\tau,\tau'\in\Sigma$, then
	\begin{align*}
d_h^{+}(\tau,\tau') = d^{+}(\tau\setminus \tau',\tau'\setminus \tau)\leq w(\tau\backslash\tau',\tau'\backslash\tau)
\end{align*}
and if $b^{-}(\tau,\tau')>0$  for $\tau,\tau'\in\Sigma$, then
\begin{align*}
d_h^{-}(\tau,\tau') = \max\{\min_{v'\in \tau'}d(\tau\setminus \tau',v'),\min_{v\in \tau}d(v,\tau'\setminus \tau)\}\leq  w(\tau\backslash\tau',\tau'\backslash\tau)
\end{align*}}
Secondly, we can write for $v,v'\in\Sigma_{0}$
\begin{equation*}
	w^{\pm}(v,v')=\inf_{\tau\cap \{v,v'\}\neq\varnothing}
	\lambda_\pm\Big(\frac{1}{m(\tau)}\sum_{\tau'\in\Sigma}b^{\pm}(\tau,\tau')\Big)^{-1/2},
\end{equation*}
where $\lambda_+ = 1,\lambda_- = 1\slash 2.$ Now observe that whenever $b^+(\tau,\tau')>0,$ then $\tau\backslash\tau'$ and $\tau'\backslash\tau$ are adjacent vertices in the 1-skeleton of $\Sigma.$ Similarly, if $b^-(\tau,\tau')>0,$ then $\tau\backslash\tau'$ and $\tau'\backslash\tau$ are vertices in $\Sigma_0$ which are connected in the 1-skeleton of $\Sigma$ through any vertex in $\tau\cap\tau'.$
With these observations it follows  that 
\begin{multline*}
	\sum_{\tau'\in \Sigma}b^{-}(\tau,\tau')d^{-}_h(\tau,\tau')^{2}\leq
		\sum_{\tau'\in \Sigma}b^{-}(\tau,\tau')d^{-}(\tau\setminus \tau',\tau\setminus\tau')^{2}
	\\
	 \leq	\sum_{\tau'\in \Sigma}b^{-}(\tau,\tau')\left(\inf_{v\in \tau\cap{\tau'}}\left(w^{-}(\tau\setminus \tau',v) +w^{-}(v,\tau\setminus\tau')\right)\right)^{2}\leq  m(\tau)
\end{multline*}
and similarly
$ \sum_{\tau'}b^{+}(\tau,\tau')d^{+}_h(\tau,\tau')^{2}\leq m(\tau)$, so $d^{\pm}_h$ are intrinsic.
Finally, since $b^{H}=|b^{+}-b^{-}|\leq b^++b^-$ and $\sqrt{2}d^H\leq d^\pm$ we see that $d^{H}_{h}$ is intrinsic as well.
 	Thus, the statement follows from Theorem~\ref{intrinsic}.   
\end{proof}

\section{Spectral Relations between Laplacians}\label{specrel}
In this section, we study several spectral relations between different Laplacians that are well known for finite complexes.
While the overall picture remains consistent,  new phenomena  occur due to the non-uniqueness 
of the Laplacian.

The first result is due to an abstract functional analytic fact that operators of the type $T^{*}T$ and $TT^{*}$ share the same spectrum apart from $0$.

\begin{theorem}\label{spectrum0}
    We have 
    \begin{align*}
        \sigma(\Delta^+_{D})\backslash\lbrace 0\rbrace &= \sigma(\Delta^-_N)\backslash\lbrace 0\rbrace,\\
        \sigma(\Delta^-_{D})\backslash\lbrace 0\rbrace &= \sigma(\Delta^+_N)\backslash\lbrace 0\rbrace,\\
        \sigma(\Delta^H_{D})\backslash\lbrace 0\rbrace &= \sigma(\Delta^H_N)\backslash\lbrace 0\rbrace.
    \end{align*}
    The same is true for the essential, pure point, absolutely continuous and singular continuous  spectrum.
\end{theorem}
\begin{proof}This follows directly 
	from \cite[Corollary~5.6.]{T} or \cite[Proposition~1.23]{FLW} and Corollary~\ref{LNLD}.
\end{proof}
\begin{remark}
	    If we restrict our attention to 
	    $\ell^2(\Sigma_k,m)$ and denote the appropriate restriction of $\Delta^\bullet_{D}$ by $\Delta^\bullet_{D,k}$ then 
	    the first two equalities in Theorem~\ref{spectrum0} become
    $$\sigma(\Delta^+_{D,k})\backslash\lbrace 0\rbrace = \sigma(\Delta^-_{N,k+1})\backslash\lbrace 0\rbrace\quad\text{and}\quad \sigma(\Delta^+_{N,k})\backslash\lbrace 0\rbrace = \sigma(\Delta^-_{D,k+1})\backslash\lbrace 0\rbrace,$$
    which is a well known relation between up- and down-Laplacians on finite simplicial complexes (see \cite{HJ}). The statement about the Hodge-Laplacian is remarkable, as to our knowledge there is no such relation for Hodge-Laplacians defined on $\ell^2(\Sigma_k,m).$ This seems to be a special feature of the operator defined across multiple dimensions. Also note that a Hodge-Laplacian with Neumann boundary condition on $\ell^2(\Sigma_k,m)$ would not be well defined as $\delta+\partial$ maps $\F(\Sigma_k)$ into $\F(\Sigma_{k+1})\oplus \F(\Sigma_{k-1}).$
\end{remark}

\eat{\begin{remark}
    After writing up this paper we got aware of the reference \cite[Corollary~5.6.]{T} that shows a stronger result than the one given in Theorem \ref{TT*} namely that the operators involved are unitarily equivalent when restricted to the orthogonal complement of their kernel. Applying this to the setting of Theorem \ref{spectrum0} yields that also the spectral types of the nonzero spectra, i.e. pure point, absolutely continuous and singular continuous spectra, coincide.
\end{remark}}

The theorem above excludes statements about $0,$ which is only natural. However, if $0$ is in $\sigma(\Delta^{H}_{D})$, then it is in all the other spectra, as the following proposition shows.

\begin{proposition}\label{spectrum2}  If $0\in\sigma(\Delta^H_{D}),$ then $0\in \sigma(\Delta^\pm_{D})\cap\sigma(\Delta^\pm_N)\cap \sigma(\Delta^{H}_{N}).$ The same is true if $0$ is in the essential spectrum of $\Delta^H_{D}.$
\end{proposition}
\begin{proof}
Assume  $0\in\sigma(\Delta^{H}_{D}) $. Since $Q^{H}_{D}\subseteq Q^{H}_{N}$, we clearly have $0\in\sigma(\Delta^{H}_{N}) $. 
	Furthermore, by Weyl's criterion we find a normalized sequence $\omega_n$ in $D(\Delta^H_{D})$ such that $\lim_n\Vert\Delta^H_{D}\omega_n\Vert = 0.$ Therefore, $Q^H_{D}(\omega_n) = \langle\Delta^H_D\omega_n,\omega_n\rangle$ tends to zero as well. However, because $Q^H_{D}$ is a restriction of $Q^+_{D}+Q^-_{D}$ and all forms are positive, we conclude that $Q^+_{D}(\omega_n)$ and $Q^-_{D}(\omega_n)$ tend to zero individually. 
	Thus, $0\in \sigma(\Delta^\pm_{D})$ and $0\in \sigma(\Delta^\pm_N)$ since $Q^{\pm}_{D}\subseteq Q^{\pm}_{N}$.
	The statement about the essential spectrum follows  the same way by choosing $\omega_n$ weakly converging to zero (\cite[Theorem~E.4]{KLW}).
\end{proof}

The final result shows that, in the case $Q^H_{D} = Q^H_N$, all nonzero spectra agree. Recall that Proposition~\ref{QH->Q+Q-} shows that $Q^H_{D} = Q^H_N$ implies $Q^\pm_{D} = Q^\pm_N$.

\begin{theorem}\label{spectrum1}  If $Q^H_{D} = Q^H_N,$ then for $\Delta^{\bullet}=\Delta^{\bullet}_{D}=\Delta_{N}^{\bullet}$, $\bullet\in\{\pm,H\}$, $$\sigma(\Delta^+)\backslash\lbrace 0\rbrace=\sigma(\Delta^H)\backslash\lbrace 0\rbrace = \sigma(\Delta^-)\backslash\lbrace 0\rbrace.$$
\end{theorem}

The following lemma is vital to the proof.
\begin{lemma}\label{q = qn}
    Let $\lambda\neq 0.$ Then 
    $$\Delta^+_N+\Delta^-_D-\lambda = -\lambda^{-1}(\Delta^+_N-\lambda)(\Delta^-_D-\lambda).$$
\end{lemma}
\begin{proof}
    Let $\omega\in D(\Delta^+_N+\Delta^-_{D}) = D(\Delta^+_N)\cap D(\Delta^-_{D}).$
    Then $\delta\Delta^-_D\omega = \delta\delta\partial\omega = 0.$ This shows $\Delta^-_D\omega\in D(\Delta^+_N)$ and $\Delta^+_N\Delta^-_D\omega = 0.$ Therefore, 
    \begin{equation*}\label{x1}
        -\lambda^{-1}(\Delta^+_N-\lambda)(\Delta^-_D-\lambda)\omega = \Delta^+_N\omega+\Delta^-_D\omega-\lambda\omega
    \end{equation*} which shows that $\Delta^+_N+\Delta^-_D-\lambda\subseteq -\lambda^{-1}(\Delta^+_N-\lambda)(\Delta^-_D-\lambda).$

     Assume $\omega\in D((\Delta^+_N-\lambda)(\Delta^-_D-\lambda)).$ This means in particular that $\omega\in D(\Delta^-_D).$ Moreover, $\Delta^-_D\omega-\lambda\omega\in D(\Delta^+_N).$ However, this tells us that $-\lambda\delta\omega = \delta(\delta\partial-\lambda)\omega\in\ell^2$ so  $\omega\in D(Q^+_N).$ Furthermore, for $\nu\in D(Q^+_N)$ we get

    $$\langle -\lambda\delta\omega,\delta\nu\rangle = \langle\delta(\Delta^-_D-\lambda)\omega,\delta\nu\rangle = \langle\Delta^+_N(\Delta^-_D-\lambda)\omega,\nu\rangle.$$ Thus, we conclude that $\omega\in D(\Delta^+_N).$ With this at hand we again have the identity $\Delta^+_N\omega+\Delta^-_D\omega -\lambda\omega= -\lambda^{-1}(\Delta^+_N-\lambda)(\Delta^-_D-\lambda)\omega$ and conclude the statement.
\end{proof}

\begin{proof}[Proof of Theorem~\ref{spectrum1}]
Let $Q^H_{D}=Q^H_N$ and $\lambda\neq 0.$ This implies $\Delta^{\pm}=\Delta^{\pm}_{D}=\Delta_{N}^{\pm}$ by Proposition~\ref{QH->Q+Q-}, as well as $\Delta^H = \Delta^++\Delta^-$ by the discussions in Section \ref{Hilbert}. Furthermore, according to Lemma~\ref{q = qn} we have $\Delta^H -\lambda= -\lambda^{-1}(\Delta^+-\lambda)(\Delta^--\lambda).$ 

If $\lambda\in\sigma(\Delta^H)$, then we find by Weyl's criterion a normalized sequence $(\omega_n)$ in $D(\Delta^H)$ such that $\Vert(\Delta^H-\lambda)\omega_n\Vert\rightarrow 0$ as $n\rightarrow\infty.$ However, this means that either $\Vert(\Delta^--\lambda)\omega_n\Vert\rightarrow 0$ as $n\rightarrow\infty$ or we find a subsequence (also denoted as $(\omega_n)$) such that $\eta_n = (\Delta^--\lambda)\omega_{n}$ is uniformly in $n$ bounded away from $0$ and $\Vert(\Delta^+-\lambda)\eta_n\Vert\rightarrow 0$ as $n\rightarrow\infty.$ We conclude with Weyl's criterion that $\lambda\in\sigma(\Delta^+)\cup\sigma(\Delta^-).$ By Theorem~\ref{spectrum0} we get $\lambda\in\sigma(\Delta^\pm).$ Thus, $\sigma(\Delta^H)\backslash\lbrace 0\rbrace\subseteq\sigma(\Delta^\pm)\backslash\lbrace 0\rbrace.$\\

Now, let $\lambda\in\sigma(\Delta^+).$
    We find a normalized sequence $\omega_n\in D(\Delta^+)$ such that 
    $\lim_n\Vert(\Delta^+-\lambda)\omega_n\Vert = 0.$ Let us define the sequence $$\mu_n = \delta\omega_n.$$ Because $\omega_n\in\mathcal{D}^+$, we find that $\mu_n\in\ell^2$ and $\Vert\mu_n\Vert^2 = \Vert\delta\omega_n\Vert^2 = Q^+(\omega_n).$ This implies that $\Vert\mu_n\Vert$ is bounded away from zero for large enough $n$ since $\lambda>0$ and
    $$Q^+(\omega_n) = \langle (\Delta^+-\lambda)\omega_n,\omega_n\rangle+\lambda.$$
    Next, we observe that $\delta\mu_n = 0$ and $\partial\mu_n = \partial\delta\omega_n = \Delta^+\omega_n.$ Consequently,
    $$\sum_\Sigma m\vert (\delta+\partial)\mu_n\vert^2 = \Vert\Delta^+\omega_n\Vert^2<\infty.$$
    This tells us that $\mu_n\in \mathcal{D}^H$ and all together $\mu_n\in D(Q^H_N).$
    However, by assumption we have that $Q^{H}=Q^H_{D} = Q^H_N.$
    We now get for $\nu\in D(Q^H_{D})$
    \begin{multline*}
    \vert Q^H(\mu_n,\nu)-\lambda\langle\mu_n,\nu\rangle\vert  = \vert\langle\delta\mu_n,\delta\nu\rangle+\langle\partial\mu_n,\partial\nu\rangle-\lambda\langle\mu_n,\nu\rangle\vert\\  
    = \vert\langle\Delta^+\omega_n,\partial\nu\rangle-\lambda\langle\delta\omega_n,\nu\rangle\vert
     = \vert\langle(\Delta^+-\lambda)\omega_n,\partial\nu\rangle\vert
    \leq \Vert(\Delta^+-\lambda)\omega_n\Vert Q^H_D(\nu)^{1/2}\to 0.
    \end{multline*}
    From \cite[Lemma 1.4.4]{StollmannDisorder} (see also \cite[Theorem~E.3]{KLW}) we conclude that $\lambda\in\sigma(\Delta^H).$ Thus $\sigma(\Delta^\pm)\backslash\lbrace 0\rbrace\subseteq\sigma(\Delta^H)\backslash\lbrace 0\rbrace$  by Theorem~\ref{spectrum0}.
\end{proof}

\begin{remark}
    For the restrictions of $\Delta^\bullet$ to $\ell^2(\Sigma_k,m)$ 
    Theorem~\ref{spectrum1} reads as $$\sigma(\Delta^H_k)\backslash\lbrace 0\rbrace =(\sigma(\Delta^+_k)\cup\sigma(\Delta^-_{k}))\backslash\lbrace 0\rbrace,$$ which is another well known relation on finite simplicial complexes. 
\end{remark}


\eat{\begin{remark}
    When dropping the assumption that $D(Q^H_D) = D(Q^H_N),$ one can still show by similar arguments as in the proof of Theorem~\ref{spectrum1} that $$\sigma(\Delta^+_N+\Delta^-_D)\backslash\lbrace 0\rbrace \subseteq\sigma(\Delta^-_{D}) \qquad\mbox{and}\qquad \sigma(\Delta^+_{D}+\Delta^-_N)\backslash\lbrace 0\rbrace \subseteq\sigma(\Delta^+_{D})$$ when $R(\delta_D)\subseteq D(\delta_D)$ for the operators introduced in Section~\ref{Hilbert}. In fact, one can show in a very similar fashion that $$\sigma(\Delta^+_N+\Delta^-_D)\backslash\lbrace 0\rbrace = \sigma(\Delta^-_{D})\backslash\lbrace 0\rbrace$$ and under the assumption that $R(\delta_D)\subseteq D(\delta_D)$ that $$\sigma(\Delta^+_{D}+\Delta^-_N)\backslash\lbrace 0\rbrace = \sigma(\Delta^+_{D})\backslash\lbrace 0\rbrace.$$
\end{remark}

}

\section{Examples  of Specific Weights} 
\subsection{Combinatorial Weight}

The weight $m$ given by $m\equiv 1$ is called the \textit{combinatorial weight}. It is by far the most natural choice, but also quite restrictive. By the local summability assumption, it already requires the complex $\Sigma$ to be locally finite, that is, every simplex is contained in only finitely many cofaces. On the other hand, many of the formulas above simplify and have interesting combinatorial interpretations. Looking at the Schr\"odinger operator representation from above, we see that the Forman curvature $c^H$ simplifies in the following way.
For $\tau\in\Sigma$ with $k = \dim(\tau)$, we have
\begin{align*}
	c^H(\tau) &= \#\lbrace\rho\prec\tau\rbrace+\#\lbrace\sigma\succ\tau\rbrace -\sum_{\tau'}b^H(\tau,\tau')\\
	& = 2(k+1)+(k+2)\#\lbrace\sigma\succ\tau\rbrace-\sum_{\rho\prec\tau}\#\lbrace \tau'\succ \rho\rbrace.
\end{align*}
Since $c^H$ is a lower bound for $Q_D^H$ every lower bound on $c^H$ yields a lower bound for the spectrum $\sigma(\Delta^H_D).$ Of course, this is only useful if $c^H$ is positive to begin with.\\
Another consequence is that $b^{\bullet}$ takes values in $\{0,1\}$
and $b^+$, $b^H$ become subgraphs of $b^-$ where $b^H$ is the complement of $b^+$ in $b^-.$

We summarize the results of Section~\ref{FormeqEssSa} for the combinatorial weight in the following theorem. To this end, we denote the vertex degree $\deg:\Sigma_{0}\to\mathbb{N}$
\begin{equation*}
	\deg(v)=\#\{e\in \Sigma_{1}\mid e\succ v\}.
\end{equation*}
Furthermore, let the degree path metric on $\Sigma_{0}$ be given by
\begin{equation*}
	\rho(v,w)=\inf_{v=v_{0}\sim\ldots\sim v_{n}=w}\sum_{k=1}^{n}\max\{\deg(v_{k}),\deg(v_{k-1})\}^{-1/2}.
\end{equation*}
Here $v\sim w$ means the usual adjacency relation on graphs, i.e., $v$ and $w$ are connected by an edge in the 1-skeleton of $\Sigma.$

\begin{theorem}[Combinatorial weight]\label{comb}
	Let $\Sigma$ be a connected finite dimensional simplicial complex with combinatorial weight $m \equiv 1$ and $\Delta^{\bullet}$ a Laplacian, $\bullet\in\{\pm,H\}$.
	\begin{itemize}
		\item [(a)] $\Delta^{\bullet}$ is bounded if and only if $\deg$ is bounded.
		\item [(b)]  $F^{c}\subseteq D(\Delta^{\bullet})$.
		\item [(c)] If the Forman curvature $c^{H}$ is bounded from below, then $\Delta^{\bullet}_{c}$ is essentially self-adjoint.
		\item [(d)] If $(\Sigma_{0},\rho)$ is a complete metric space, then $\Delta^{\bullet}_{c} $ is essentially self-adjoint.		
	\end{itemize}
\end{theorem}
\begin{proof} \emph{(a)} We have  for $\tau\in\Sigma,v\in\tau$
			$$\gamma^{+}(\tau)=\#\{\sigma\succ \tau\} \leq\deg(v)
$$ so boundedness follows from Theorem~\ref{boundedness}.

\emph{(b)} This is clear due to local finiteness and Theorem~\ref{Fc}.

\emph{(c)} This follows directly from Theorem~\ref{esssa1}.

\emph{(d)} We first observe that  finite metric balls are equivalent to completeness for locally finite path metrics, see \cite[Theorem~11.16]{KLW} or \cite[Theorem~2.1.]{KM}.
We define a path pseudo metric on $\Sigma_k,$ $0\leq k\leq\dim(\Sigma)$, as follows. Let $$w(\sigma,\sigma') = ((\dim(\Sigma)+1)\min_{v\in\sigma\cap\sigma'}\deg(v))^{-1/2}$$ whenever $b^H(\sigma,\sigma')>0$ (we  write $\sigma\sim\sigma'$) and $w(\sigma,\sigma') = 0$ otherwise. Then 
$$d(\sigma,\sigma') =\inf_{\sigma=\sigma_{0}\sim\ldots\sim \sigma_{n}=\sigma'}\sum_{k=1}^{n}w(\sigma_k,\sigma_{k-1}) $$
defines a path pseudo metric. Moreover, $d$ is intrinsic for $b^H$ as for $\sigma\in\Sigma_k$
\begin{multline*}
    \sum_{\sigma\in \Sigma} b^H(\sigma,\sigma')d^2(\sigma,\sigma')\leq \sum_{\rho\prec\sigma}\sum_{\sigma\neq\sigma'\succ\rho}((\dim(\Sigma)+1)\min_{v\in\rho}\deg(v))^{-1} \\\leq\sum_{\rho\prec\sigma}\frac{\gamma^+(\rho)}{(\dim(\Sigma)+1)\min_{v\in\rho}\deg(v)}\leq 1.
\end{multline*}
We continue the proof with two claims

\emph{Claim 1.} For all $\sigma,\sigma'\in\Sigma_k$, we find $v\in\sigma,v'\in\sigma'$ such that
$$\rho(v,v')\leq  (\dim(\Sigma)+1)^{1/2} d(\sigma,\sigma').$$ 
\emph{Proof of the claim.}
Take an arbitrary path between $\sigma$ and $\sigma',$ i.e., $\sigma = \sigma_{0}\sim\ldots\sim\sigma_n = \sigma'.$ Now let $v_k$ be any vertex in $\sigma_k\cap\sigma_{k+1},$ which is non empty as $b^H(\sigma_k,\sigma_{k+1})>0.$ For any $k = 0,\ldots,n-1$, we either have $v_k = v_{k+1}$ or $v_k$ and $v_{k+1}$ are connected by an edge in the 1-skeleton of $\Sigma.$ This is because both are contained in $\sigma_{k+1}.$ That way we obtain a path in the 1-skeleton of $\Sigma$ that connects $v_0\in\sigma$ and $v_n\in\sigma'.$ The claim follows as 
$$(\dim(\Sigma)+1)^{1/2}\sum_{k=1}^{n}w(\sigma_k,\sigma_{k-1})\geq\sum_{k=1}^{n-1}\deg(v_{k-1})^{-1/2}\geq \rho(v,v').$$

\emph{Claim 2.} If  $\rho$ has finite distance balls then $d$ has finite distance balls as well.

\emph{Proof of the claim.}  Assume there is an infinite distance ball of radius $r$ centered at some $\sigma_{0}\in \Sigma$ for $d$. By local finiteness there are infinitely many disjoint $\sigma_{n}\in \Sigma$ such that  for $ C=(\dim(\Sigma)+1)^{1/2}$
 $$Cr\ge Cd(\sigma_{0},\sigma_{n})\geq \rho(v_{0},v_{n})$$
with vertices $v_{0}\in \sigma_{0}$ and $v_{n}\in \sigma_{n}$ by the claim above. As $\sigma_{0}$ contains only finitely many vertices, there is an infinite distance ball for $\rho,$ which settles the claim.

By Claim~2 and since $d$ is intrinsic, we get essential self-adjointness of $\Delta^H_c$ from Proposition~\ref{intrinsic}. The statement for $\Delta^\pm_c$ then follows from Theorem~\ref{esssa2}.\end{proof}

\subsection{Normalizing Weight}
In the case of weighted graphs, one typically defines the normalizing weight $m$ on vertices as $m(x) = \sum_{e_x}m(e_x),$ where the sum runs over all edges $e_x$ that contain the vertex $x$. This has two very useful effects. First, it makes the graph Laplacian $\Delta$ bounded such that the spectrum is contained in the interval $[0,2].$ 
The second effect is that $\Delta$ takes the form $\Delta = I-P$ where $P$ is the transition matrix, which defines the standard random walk on the graph.  This means that the kernel $p(x,y)$ of $P$ can be seen as the probability of jumping from a vertex $x$ to a vertex $y.$ 
In \cite[Theorem 3.2.]{HJ} the weight
\begin{equation*}
	 m^{+}(\tau) = \sum_{\sigma\succ\tau}m^{+}(\sigma)
\end{equation*}
for $\tau\in\Sigma$, with $m^{+}$ arbitrary on the maximal faces, was already considered and it was shown that the spectrum on $\Sigma_{k}$ is included in $[0,k+2]$. Here, we discuss that a similar result holds for infinite complexes.

\begin{theorem}For $m^{+}$, all operators $\Delta^{\bullet}$, $\bullet\in\{\pm,H\}$, are bounded if $\Sigma$ is finite dimensional. Furthermore, the restrictions to $\Sigma_{k}$ satisfy  $\sigma(\Delta^\bullet_k)\subseteq [0,k+2].$  Moreover, if every $\tau\in\Sigma_k$ has a coface, then $\Delta^+_k$ acts on $\ell^2(\Sigma_k,m)$ as $$\Delta^+_k\omega(\tau) = \omega(\tau)-\frac{1}{m^+(\tau)}\sum_{\tau'\in\Sigma}b^+(\tau,\tau')\theta^+(\tau,\tau')\omega(\tau').$$
\end{theorem}
\begin{proof}Boundedness follows as $\gamma^{+}/m^{+}=1$ by Theorem~\ref{bounddim}. Indeed, as $\Sigma^{N}=\bigcup_{k\leq N}\Sigma_{k}$ is a subcomplex of $\Sigma,$ all restrictions are bounded. Let $m^+_{k},\omega_{k}$ be the restrictions of $m^+$ and $\omega $ to $\Sigma_{k}$ for some $\omega\in F$. The statement about the spectrum follows by 
\begin{equation*}
	\|\delta_{k} \omega_k\|^{2}=\sum_{\Sigma_{k+1}}m^{+}_{k+1}|\delta_{k} \omega|^{2}\leq (k+2) \sum_{\Sigma_{k}}m^{+}_k|\omega_k|^{2},
\end{equation*}
where we used the definition of $m^{+}$ and the fact that every $k+1$ simplex has $k+2$ faces. This yields $\sigma(\Delta^+_k)\subseteq [0,k+2]. $ The respective statement for $\Delta^-_k$ and $\Delta^H_k$ follow from Theorem~ \ref{spectrum1}~and Theorem~\ref{spectrum0} and the remarks thereafter.
The proof about the action of $\Delta^+_k$ is just a simple and straight forward calculation. 
\end{proof}
Observe that while we achieve to write $\Delta^{+}=I-\tilde P$, the operator $\tilde P$ is not a stochastic matrix and not even positive anymore, which makes a probabilistic interpretation a challenging task, see e.g. \cite{Parzanchevski}.

Note that the particular choice of weight in the theorem above makes the diagonal of $\Delta^+_k$ constant, provided every element $\tau\in\Sigma_k$ has a coface. For $\Delta^-_k$, the weight
$$m(\tau) =  \Big(\sum_{\rho\prec\tau}\frac{1}{m(\rho)}\Big)^{-1}$$ achieves the same, although in general one does not obtain a bounded operator in this way, cf. Theorem~\ref{boundedness}. Note that the particular choice of $m$ for $\Delta^+_k$ depends on the weights on $\Sigma_{k+1}$ where for $\Delta^-_k$ it depends on the weights on $\Sigma_{k-1}.$ 

Therefore it is not surprising that in order to have constant diagonal of $\Delta^H_k$ one needs to control the weights on both $\Sigma_{k+1}$ and $\Sigma_{k-1}.$ To see this, note that we want 
$$m(\tau) = \sum_{\tau'\in \Sigma}b^H(\tau,\tau')+c^H(\tau) = \sum_{\rho\prec\tau}\frac{m^2(\tau)}{m(\rho)}+\sum_{\tau\prec\sigma}m(\sigma)$$ or equivalently 
$$0 = m^2(\tau)\sum_{\rho\prec\tau}\frac{1}{m(\rho)}-m(\tau)+\sum_{\sigma\succ \tau}m(\sigma).$$ This polynomial in $m(\tau)$ has positive real solutions, if the coefficients $\sum_{\rho\prec\tau}\frac{1}{m(\rho)}$ and $\sum_{\sigma\succ \tau}m(\sigma)$
are small enough. In this case, one obtains an operator $\Delta^{H}_{k}$ that has constant diagonal.


\bibliographystyle{alpha}
\bibliography{literature}

\end{document}